\documentclass[12pt,reqno]{amsart}
\usepackage{amsfonts, amsthm, amsmath}
\allowdisplaybreaks[4]

\usepackage{rotating}

\usepackage{tikz}

\usepackage{graphics}

\usepackage{amssymb}

\usepackage{amscd}

\usepackage[latin2]{inputenc}

\usepackage{t1enc}

\usepackage[mathscr]{eucal}

\usepackage{indentfirst}

\usepackage{graphicx}

\usepackage{graphics}

\usepackage{pict2e}

\usepackage{mathrsfs}

\usepackage{enumerate}
\usepackage[pagebackref]{hyperref}
\hypersetup{colorlinks=true}
\usepackage{cite}
\usepackage{color}
\usepackage{epic}
\usepackage{pdfsync} 
\numberwithin{equation}{section}
\theoremstyle{plain}

\newtheorem{thm}{Theorem}[section]
\newtheorem{lem}[thm]{Lemma}

\newtheorem{cor}{Corollary}[section]

\theoremstyle{definition}

\newtheorem{example}[thm]{Example}

\newtheorem{remark}[thm]{Remark}

\newtheorem{?}[thm]{Problem}

\def\N{\mathbb{N}}
\def\Z{\mathbb{Z}}
\def\S{\mathfrak{S}}
\def\D{\mathfrak{D}}
\def\L{\mathcal{LH}}

\def\B{\mathcal{B}}

\def\valley{\mathsf{val}}
\def\val{\mathsf{val}}
\def\peak{\mathsf{pk}}
\def\da{\mathsf{da}}
\def\dd{\mathsf{dd}}
\def\cvalley{\mathsf{cval}}
\def\cval{\mathsf{cval}}
\def\cpeak{\mathsf{cpk}}
\def\cpk{\mathsf{cpk}}
\def\cda{\mathsf{cda}}
\def\cdd{\mathsf{cdd}}

\def\lpeak{\mathsf{lpk}}
\def\lvalley{\mathsf{lval}}
\def\ldd{\mathsf{ldd}}
\def\lda{\mathsf{lda}}
\def\rpeak{\mathsf{rpk}}
\def\rvalley{\mathsf{rval}}
\def\rda{\mathsf{rda}}
\def\rdd{\mathsf{rdd}}

\def\sde{\mathsf{scda}}

\def\des{\mathsf{des}}
\def\asc{\mathsf{asc}}
\def\exc{\mathsf{exc}}
\def\drop{\mathsf{drop}}
\def\wex{\mathsf{wex}}
\def\inv{\mathsf{inv}}
\def\fix{\mathsf{fix}}
\def\cros{\mathsf{cros}}
\def\nest{\mathsf{nest}}
\def\icr{\mathsf{icr}}

\def\cab{31\text{-}2}
\def\bca{2\text{-}31}
\def\acb{13\text{-}2}
\def\bac{2\text{-}13}

\def\wexa{\mathsf{wex_A}}
\def\exca{\mathsf{exc_A}}
\def\wexc{\mathsf{wex_C}}

\def\fmax{\mathsf{fmax}}
\def\neg{\mathsf{neg}}
\def\Orb{\mathsf{Orb}}

\def\stat{\mathsf{stat}}
\def\cyc{\mathsf{cyc}}

\def\fmin{\mathsf{fmin}}
\def\amax{\mathsf{amax}}
\def\amin{\mathsf{amin}}
\def\fmax{\mathsf{fmax}}
\def\amax{\mathsf{amax}}

 \def\2MP{\mathsf{2\text{-}M}}

 \def\LH{\mathcal{LH}}
 \def\su{\textsc{U}}
 \def\sd{\textsc{D}}
 \def\lr{\textsc{L}_r}
 \def\la{\textsc{L}_r}
 \def\lb{ \textsc{L}_b}
 
 \def\ly{\textsc{L}_y}
 \def\lam{\lambda} 
\def\stan{\textrm{stan}}
 \DeclareMathOperator\Cval{\mathrm{Cval}}
\DeclareMathOperator\Cpeak{\mathrm{Cpk}}
\DeclareMathOperator\Cda{\mathrm{Cda}}
\DeclareMathOperator\Fix{\mathrm{Fix}}
\DeclareMathOperator\Cdd{\mathrm{Cdd}}
\def\DD{\mathrm{DD}\,}
\def\DE{\mathrm{DE}\,}
\def\SDE{\mathrm{SDE}\,}

\DeclareMathOperator\Wex{\mathrm{Wex}}

\DeclareMathOperator\Lval{\mathrm{Lval}}
\DeclareMathOperator\Lpeak{\mathrm{Lpk}}
\DeclareMathOperator\Lda{\mathrm{Lda}}
\DeclareMathOperator\Ldd{\mathrm{Ldd}}
\DeclareMathOperator\Val{\mathrm{Val}}
\DeclareMathOperator\Da{\mathrm{Da}}
\DeclareMathOperator\Dd{\mathrm{Dd}}
\DeclareMathOperator\Peak{\mathrm{Pk}}

\DeclareMathOperator\Drop{\mathrm{Drop}}
\DeclareMathOperator\Scval{\mathrm{Scval}}
\DeclareMathOperator\Scpeak{\mathrm{Scpk}}
\DeclareMathOperator\Sde{\mathrm{Scda}}
\DeclareMathOperator\Sdn{\mathrm{Scdn}}

\DeclareMathOperator\Exc{\mathrm{Exc}}

\def\s{\mathbf{s}}

\begin{document}
\title[Eulerian polynomials and excedance statistics]
{Eulerian polynomials and excedance statistics}
\author[B. Han]{Bin Han}

\address[Bin Han]{Univ Lyon, Universit\'e Claude Bernard Lyon 1, CNRS UMR 5208, Institut Camille Jordan, 43 blvd. du 11 novembre 1918, F-69622 Villeurbanne cedex, France}
\email{han@math.univ-lyon1.fr, han.combin@hotmail.com.}

\address[Bin Han]{{\itshape Current address}: 
Department of Mathematics, Bar-Ilan University, Ramat-Gan 52900, Israel.}
\email{han@math.biu.ac.il, han.combin@hotmail.com.}

\author[J. Mao]{Jianxi Mao}
\address[Jianxi Mao]{School of Mathematic Sciences, 
 Dalian University of Technology, Dalian 116024, P. R. China}
\email{maojianxi@hotmail.com}

\author[J. Zeng]{Jiang Zeng}

\address[Jiang Zeng]{Univ Lyon, Universit\'e Claude Bernard Lyon 1, CNRS UMR 5208, Institut Camille Jordan, 43 blvd. du 11 novembre 1918, F-69622 Villeurbanne cedex, France}
\email{zeng@math.univ-lyon1.fr}

\date{\today}

\begin{abstract} A formula of Stembridge states   that 
the permutation peak polynomials and descent  polynomials are connected 
via a quadratique transformation. The aim of this paper is to establish the cycle analogue of Stembridge's formula by using cycle peaks and excedances of permutations. We prove a series of new general formulae expressing polynomials counting permutations by various excedance statistics in terms of refined Eulerian polynomials.  Our formulae  are comparable with   Zhuang's generalizations [Adv. in Appl. Math. 90 (2017) 86-144] using descent  statistics of permutations.  
Our methods include permutation enumeration techniques 
involving variations of classical bijections from permutations to Laguerre histories,
 explicit continued fraction expansions of combinatorial generating functions in 
 Shin and Zeng~[European J. Combin. 33 (2012), no. 2, 111--127] 
 and cycle version of modified Foata-Strehl action. 
We also prove  similar  formulae for restricted permutations such as derangements and permutations avoiding certain patterns. Moreover, we provide new combinatorial interpretations for the $\gamma$-coefficients of the inversion polynomials restricted on $321$-avoiding permutations.

\end{abstract}



\maketitle

\tableofcontents

\section{Introduction}\label{sec1: Intro}
Stieltjes~\cite{Sti89} showed that the Eulerian polynomials $A_n(t)$ can be defined 
through the continued fraction (S-fraction) expansion
\begin{align}\label{gf:eulerian}
 \sum_{n\geq 0}A_n(t)z^n=\cfrac{1}{1-\cfrac{1\cdot z}{1-\cfrac{t\cdot z}{1-\cfrac{2\cdot z}{1-\cfrac{2t\cdot z}{1-\ldots}}}}}.
   \end{align}
For  an $n$-permutation 
$\sigma:=\sigma(1)\sigma(2)\cdots\sigma(n)$ of the word $1\ldots n$,
an index  $i$ ($1\leq i\leq n-1$) is a \textit{descent} (resp. \textit{excedance}) of $\sigma$
 if $\sigma(i)>\sigma(i+1)$ (resp. $\sigma(i)>i$).   
 It is well-known~\cite{FS70, Pet15} that   
\begin{align}\label{des-exc}
A_{n}(t)=\sum_{\sigma\in\mathfrak{S}_{n}}t^{\des\, \sigma}=\sum_{\sigma\in\mathfrak{S}_{n}}t^{\exc\,\sigma},
\end{align}
where $\mathfrak{S}_{n}$ is  the set of $n$-permutations and 
$\des\,\sigma$ (resp. $\exc\,\sigma$) denotes the number of descents (resp. excedances) of $\sigma$.
The value  $\sigma(i)$ ($2\leq i\leq n-1$) is a \textit{peak} of
$\sigma$ if $\sigma(i-1)<\sigma(i)>\sigma(i+1)$ and 
the \emph{peak polynomials} are defined by 
\begin{equation}\label{eq:peak-poly}
P_{n}^{\peak}(x):=\sum_{\sigma\in\mathfrak{S}_{n}}x^{\peak'\,\sigma},
\end{equation}
where  $\peak'\,\sigma$ denotes the number of peaks of $\sigma$.
The peak polynomials are related to the Eulerian polynomials 
by Stembridge's identity~\cite[Remark 4.8]{Ste97}, see also \cite{Bra08, Zhuang17}, 
\begin{equation}\label{eq:stembridge}
A_{n}(t)=\left(\frac{1+t}{2}\right)^{n-1}P_{n}^{\peak}\left(\frac{4t}{(1+t)^{2}}\right),
\end{equation}
which can be used  to compute the peak polynomials.
Obviously Eq. \eqref{eq:stembridge}  is equivalent to the so-called $\gamma$-expansion of Eulerian polynomials
\begin{equation}\label{FSGP}
A_n(t)=\sum_{k=0}^{\lfloor (n-1)/2\rfloor} 2^{2k+1-n}\gamma_{n,k}t^k(1+t)^{n-1-2k},
\end{equation}
where $\gamma_{n,k}$ is the number of $n$-permutations with $k$ peaks.
In  the form of \eqref{FSGP} it is not difficult to see that 
Stembridge's formula~\eqref{eq:stembridge}  is actually equivalent to 
a formula of Foata and Sch\"{u}zenberger~\cite[Th\'eor\`eme 5.6]{FS70}, 
see  also  Br\"{a}nd\'{e}n's proof using \emph{modified Foata-Strehl action}~\cite{Bra08}.
In the last two decades, 
many  refinements of Stembridge's identity have been given by 
Br\"{a}nd\'{e}n~\cite{Bra08},  Petersen~\cite{Pet07}, 
Shin and Zeng~\cite{SZ12, SZ16}, Zhuang~\cite{Zhuang17},
Athanasiadis~\cite{Ath18} and others.
In particular, Zhuang~\cite{Zhuang17} has proved 
several identities expressing polynomials counting permutations by various descent statistics in terms of Eulerian polynomials, extending results of Stembridge, Petersen and Br\"and\'en.   

By contracting the  continued fraction~\eqref{gf:eulerian}  
starting from the  first and second lines (see Lemma~\ref{contra-formula}), respectively,
we derive   the two Jacobi-type continued fraction (J-fraction)  formulae (cf. \cite{Fla80})
\begin{subequations}
\begin{align}\label{lag1}
\sum_{n\geq 0} A_{n+1}(t) z^n
=\cfrac{1}{1-(1+t)\cdot z-\cfrac{1\cdot2\cdot t\cdot z^{2}}{1-2(1+t)\cdot z-\cfrac{2\cdot3\cdot t\cdot z^{2}}{1-3(1+t)\cdot z-\cfrac{3\cdot4\cdot t z^{2}}{1-\cdots}}}},
\end{align}
and 
\begin{align}\label{lag2}
\sum_{n\geq 0} A_n(t) z^n=\cfrac{1}{1-(1+0\cdot t)\cdot z-\cfrac{1^2\cdot z^{2}}{1-(2+1\cdot t)\cdot z-\cfrac{2^2\cdot t\cdot  z^{2}}{1-(3+2\cdot t)\cdot z-\cfrac{3^2\cdot t\cdot z^{2}}{1-\cdots}}}}.
\end{align}
\end{subequations}

In view of  Flajolet's combinatorial theory
 for generic J-type continued fraction expansions~\cite{Fla80}, 
 Fran\c con-Viennot's bijection $\psi_{FV}$ (resp. its restricted version $\phi_{FV}$) between permutations and \emph{Laguerre histories} provides  a bijective proof of \eqref{lag1} (resp. \eqref{lag2}), while 
Foata-Zeilberger's bijection $\psi_{FZ}$~\cite{FZ90} 
gives a bijective proof of  \eqref{lag2}. 
More precisely,
Fran\c con-Viennot~\cite{FV79}  set up 
a bijection (and its restricted version) 
from permutations to \emph{Laguarre histories} using  
\emph{linear statistics of permutation}, while 
Foata-Zeilberger constructed  another bijection~\cite{FZ90}  using 
\emph{cyclic statistics of permutations}. In 1997
Clarke-Steingr\'imsson-Zeng~\cite{CSZ97} constructed a  bijection  $\Phi$
on permutations  converting  statistic $\des$ into $\exc$  on permutations 
 and linking  the restricted  Fran\c con-Viennot's bijection $\phi_{FV}$ to  Foata-Zeilberger bijection $\phi_{FZ}$, see the right diagram in Figure~\ref{FVBFZB}. Later, similar to  $\Phi$, Shin and Zeng~\cite{SZ12} constructed  a bijection $\Psi$ on permutations to convert  linear statistics to cycle statistics on permutations corresponding to \eqref{lag1}. We will show  that 
 the composition  $\psi_{FV}\circ  \Psi^{-1}$ coincides with   
 a recent bijection of Yan, Zhou and Lin~\cite{YZL19}, see Figure~\ref{FVBFZB}.

%

\begin{figure}
\begin{center}

\begin{tikzpicture}[scale=0.3]

\node (LH) at (-12,10) {$\LH_n$};

\node (SL) at (-20,0) {$\S_{n+1}$};

\node (SC) at (-5,0) {$\S_{n+1}$};

\node (1) at (-18, 5){$\psi_{{FV}}$};

\node (2) at (-6, 5){$\psi_{{YZL}}$};

\node (3) at (-12, 1) {$\Psi$}; 
\begin{scope}[font=\footnotesize,->]

\draw (SL) -- (LH);

\draw (SC) -- (LH);

\draw (SL) --(SC);
\end{scope}

\end{tikzpicture}
\hspace{2cm}
\begin{tikzpicture}[scale=0.3]

\node (LH) at (-12,10) {$\LH^*_n$};

\node (SL) at (-20,0) {$\S_{n}$};

\node (SC) at (-5,0) {$\S_{n}$};

\node (1) at (-18, 5){$\phi_{{FV}}$};

\node (2) at (-6, 5){$\phi_{{FZ}}$};

\node (3) at (-12, 1) {$\Phi$}; 
\begin{scope}[font=\footnotesize,->]

\draw (SL) --  (LH);

\draw (SC) -- (LH);

\draw (SL) -- (SC);

\end{scope}
\end{tikzpicture}

\end{center}
\caption{\label{FVBFZB}
Two factorizations:  $ \psi_{FV}= \psi_{YZL}\circ \Psi$ and  $\phi_{FV}=\phi_{FZ}\circ \Phi$.  
}
\end{figure}
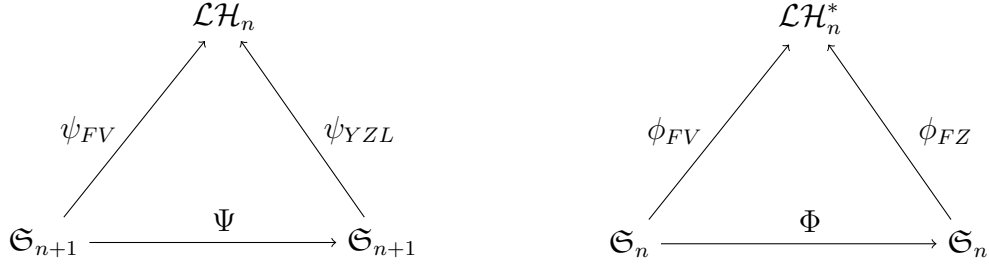

The Narayana polynomials $N_n(t)$ can be defined by the S-fraction expansion
\begin{align}\label{cf:narayana}
\sum_{n\geq 0} N_n(t) z^n= \cfrac{1}{1-\cfrac{z}{1-\cfrac{t\cdot z}{1-\cfrac{z}{1-\cfrac{t\cdot z}{1-\cdots}}}}},
\end{align}
see \cite{FTHZ19}. Note that $N_n(1)$ is the n-th Catalan number $C_n=\frac{1}{n+1}{2n\choose n}$. 
Similar to Eulerian polynomials, by contracting the S-fraction~\eqref{cf:narayana}
we derive immediately  the followoing J-fractions 
\begin{subequations}
\begin{align}\label{2-motzkin}
\sum_{n\geq 0} N_{n+1}(t) z^n=\cfrac{1}{1-(1+t)\cdot z-\cfrac{t\cdot z^{2}}{1-(1+t)\cdot z-\cfrac{
 t\cdot z^{2}}{1-\cdots}}}
\end{align}
and 
\begin{align}\label{2-motzkin-star}
\sum_{n\geq 0} N_n(t) z^n=\cfrac{1}{1-z-\cfrac{ t\cdot z^{2}}{1-(1+t)\cdot z-\cfrac{t\cdot z^{2}}{1-(1+t)\cdot z-\cdots}}}.
\end{align}
\end{subequations}
 
Let $\tau\in \S_3=\{123, 132, 213, 231, 312, 321\}$.   Recall that 
a permutation  $\sigma\in \S_n$ is said to avoid the pattern $\tau$  if 
there is no triple of indices $i<j<k$ such that  $\sigma(i)\sigma(j)\sigma(k)$ is order-isomorphic to $\tau$.  
We shall write $\S_n(\tau)$ for the set of permutations in $\S_n$ avoiding the pattern $\tau$.
It is known~\cite{FTHZ19} that the Narayana polynomials have the combinatorial interpretations
\begin{align}\label{combi-narayana}
N_n(t)=\sum_{\sigma\in \S_n(231)}t^{\des\,\sigma}=\sum_{\sigma\in \S_n(321)}t^{\exc\,\sigma}.
\end{align}
Hence, the Narayana polynomials can be considered as the Eulerian polynomials for 
restricted permutations. Moreover, they  are $\gamma$-positive and 
 have the  $\gamma$-expansion~\cite[Chapter 4]{Pet15}:
\begin{align}\label{nara-gamma}
N_n(t)=\sum_{j=0}^{(n-1)/2}\widetilde{\gamma}_{n,j}t^j (1+t)^{n-1-2j},
\end{align}
where $\widetilde{\gamma}_{n,j}=|\{\sigma\in \S_n(231): \;\des(\sigma)=\peak(\sigma)=j\}|$.

In this paper we shall   give generalizations of 
Stembridge's formula or their $\gamma$-analogues \eqref{FSGP} and \eqref{nara-gamma} using  excedance  statistics 
by further exploiting the continued fraction technique in~\cite{SZ10, SZ12, SZ16}.
Indeed,  from the observation (cf. \cite{SZ12})  that the gamma-positive formula
 of Eulerian polynomials \eqref{FSGP}
is equal to  the Jacobi-Rogers polynomial  corresponding to \eqref{lag1}, 
it becomes clear that Flajolet-Viennot's combinatorial theory of formal 
 continued fractions could shed more lights  on this topic. Our main tool is the combinatorial theory  of continued fractions due to
Flajolet~\cite{Fla80} and  bijections due to Fran\c{c}on-Viennot, Foata-Zeilberger between
permutations and \emph{Laguarre histories}, see \cite{FV79, FZ90, Fla80, CSZ97, SZ10}. As in \cite{SZ12} this approach  uses both linear and cycle statistics on permutations. 
There are several well-known 
$q$-Narayana polynomials in the litterature; see \cite{FTHZ19} and the references therein. As a follow-up to \cite{FTHZ19}, we shall give more results on $q$-Narayana
polynomials using pattern-avoiding permutations.  

The rest of this paper is organized as follows.  In Section~2 after recalling the 
necessary definitions and results from \cite{SZ10,SZ12,SZ16},
 we link the recent bijection $\psi_{YZL}$   
of Yan-Zhou-Lin~\cite{YZL19} to two known bijections;
in Section~3 we present our generalized formulae of \eqref{eq:stembridge}
in three classes:
\begin{itemize} 
\item Eulerian polynomials for permutations and derangements, 
\item Eulerian polynomials for pattern-avoiding permutations,
\item Eulerian polynomials for signed permutations.
\end{itemize}
There are  two types of proof: group actions of Foata-Strehl's type and manipulations of continued fractions. More precisely,
 we prove Theorems~\ref{thm:SZ3q} and \ref{thm:cpeakS_1}
using variations of modified Foata-Strehl  action on permutations or 
 Laguerre histories in Section~4; we then prove 
these two theorems  and 
the remaining theorems by comparing  the continued fraction 
expansions of the generating functions  in Section~5. 
In what follows, we shall abbreviate  "generating functions" by "g.f.".

\section{Background and preliminaries}
		\subsection{Permutation statistics and two bijections}	
	
For  $\sigma=\sigma(1)\sigma(2)\cdots\sigma(n)\in \S_{n}$  with convention 0--0, i.e., $\sigma(0)=\sigma(n+1)=0$,  a value $\sigma(i)$ ($1\leq i\leq n$) is called
\begin{itemize}
\item a \emph{peak} if $\sigma(i-1)<\sigma(i)$ and $\sigma(i)>\sigma(i+1)$;
\item a \emph{valley} if $\sigma(i-1)>\sigma(i)$ and $\sigma(i)<\sigma(i+1)$;
\item a \emph{double ascent} if $\sigma(i-1)<\sigma(i)$ and $\sigma(i)<\sigma(i+1)$;
\item a \emph{double descent} if $\sigma(i-1)>\sigma(i)$ and $\sigma(i)>\sigma(i+1)$.
\end{itemize}

The set of peaks (resp. valleys, double ascents, double descents) of $\sigma$
is denoted by 
$$\Peak\sigma \quad  (\text{resp.} \; \Val\sigma,\; \Da\sigma,\;\Dd\sigma). $$
Let $\peak\,\sigma$ (resp. $\valley\,\sigma$, $\da\,\sigma$, $\dd\,\sigma$) 
be the number of peaks (resp. valleys, double ascents, double descents) of $\sigma$. 
For  $i\in [n]:=\{1, \ldots, n\}$, we introduce the following 
statistics:
\begin{equation}\label{31-2}
\begin{array}{ll}
(\cab)_i\:\sigma&=\#\{j:  1<j<i \text{ and } \sigma(j)<\sigma(i)<\sigma(j-1)\}\\
(\bca)_i\:\sigma&=\#\{j:  i<j<n \text{ and } \sigma(j+1)<\sigma(i)<\sigma(j)\}\\
(\bac)_i\:\sigma&=\#\{j: i<j<n \text{ and } \sigma(j)<\sigma(i)<\sigma(j+1) \}\\
(\acb)_i\:\sigma&=\#\{j: 1<j<i \text{ and } \sigma(j-1)<\sigma(i)<\sigma(j) \}
\end{array}
\end{equation}
and define four statistics (see \eqref{def:vincular}):
$$
(\cab)=\sum_{i=1}^n(\cab)_i, \quad
(\bca)=\sum_{i=1}^n(\bca)_i, \quad
(\bac)=\sum_{i=1}^n(\bac)_i, \quad
(\acb)=\sum_{i=1}^n(\acb)_i.
$$

Now, we consider $\sigma\in \S_{n}$ as a bijection 
$i\mapsto \sigma(i)$ for $i\in [n]$, a  value $x=\sigma(i)$  is called
\begin{itemize}
	\item a \emph{cyclic peak}  if $i=\sigma^{-1}(x)<x$ and $x >\sigma(x)$;
	\item a \emph{cyclic valley}  if $i=\sigma^{-1}(x)>x$ and $x<\sigma(x)$;
	\item a \emph{double excedance} if $i=\sigma^{-1}(x)<x$ and $x<\sigma(x)$; 
	\item a  \emph{double drop} if $i=\sigma^{-1}(x)>x$ and $x>\sigma(x)$;
	\item a \emph{fixed point}  if $x=\sigma(x)$.
\end{itemize}

We say that $i\in[n-1]$ is an \textit{ascent}  of $\sigma$
 if $\sigma(i)<\sigma(i+1)$ 
 and that $i\in[n]$ is a \textit{drop} 
  of $\sigma$
if $\sigma(i)<i$.
Let 
\begin{align}\label{def:cval}
\Cpeak\quad \textrm{(resp. $\Cval$, $\Cda$,  $\Cdd$,  $\Fix$, $\Drop$})
\end{align}
 be  the set  of $\emph{cyclic peaks}$
(resp.  $\emph{cyclic valleys}$, $\emph{ double excedances}$, $\emph{double drops}$, $\emph{fixed points}$, $\emph{drops}$) and 
denote the corresponding cardinality by  
$\cpeak$ (resp. $\cval$, $\cda$, $\cdd$, $\fix$, $\drop$).
Obviously we have 
\begin{align}\label{cpk=cval}
\cpeak \,\sigma=\cvalley\, \sigma\quad \text{for}\quad  \sigma\in \S_n.
 \end{align}
Moreover, we define 
\begin{subequations}
\begin{align}
\wex\:\sigma &=\#\{i:i\leq\sigma(i)\}=\exc\:\sigma+\fix\:\sigma\\
\cros_i\:\sigma &=\#\{j:\:j<i<\sigma(j)<\sigma(i)\;\textrm{or}\;\sigma(i)<\sigma(j)\leq i<j\},\label{def:cros}\\
\nest_i\:\sigma &=\#\{j:\:j<i<\sigma(i)<\sigma(j)\;\textrm{or}
\;\sigma(j)<\sigma(i)\leq i<j\}.\label{def:nest}
\end{align} 
\end{subequations}
Let 
$\cros =\sum_{i=1}^n\cros_i$, $\nest=\sum_{i=1}^n\nest_i$ and  $\icr\:\sigma\footnotemark =\cros\:\sigma^{-1}$.
\footnotetext{Our definition of $\cros$ 
corresponds to $\icr$ in \cite{FTHZ19}.}
 Note (cf  \cite[Remark 2.4]{FTHZ19})
that 
 \begin{align}\label{nest=nest-1}
 \nest\:\sigma^{-1}=\nest\:\sigma\quad \text{for}\quad  \sigma\in \S_n.
 \end{align}
A pair of integers $(i,j)$ is an \emph{inversion} of $\sigma\in \S_n$ if 
$i<j$  and $\sigma(i)>\sigma(j)$, and    $\sigma(i)$ (resp. $\sigma(j)$) 
is called  \emph{inversion top} (resp. \emph{bottom}).
Let $ \inv\:\sigma $ be 
the inverion number of $\sigma$.

 For $\sigma\in \S_n$ with convention 0--$\infty$, i.e., 
  $\sigma(0)=0$ and $\sigma(n+1)=\infty$,
          let 
           $$
           \Lpeak \quad \textrm{(resp. $\Lval$, $\Lda$, $\Ldd$})$$
            be  the set of $\emph{peaks}$
           (resp.  $\emph{valleys}$,         $\emph{double ascents}$ and $\emph{double decents}$) and 
           denote the corresponding cardinality by  $\lpeak$
           (resp. $\lvalley$, $\lda$ and $\ldd$).
          For $i\in [n]$, the value $\sigma(i)$ is called a {\em left-to-right maximum} (resp. {\em right-to-left minimum}) if $\sigma(i)=\textrm{max}\:\{\sigma(1),\sigma(2),\ldots,\sigma(i)\}$ (resp. $\sigma(i)=\textrm{min}\:\{\sigma(i), \ldots, \sigma(n-1),\sigma(n)\}$). 
          Similarly, we define {\em left-to-right minimum} (resp. {\em right-to-left maximum}).
          
          A double ascent $\sigma(i)$ $(i=1, \ldots, n)$ is called a \emph{foremaximum} (resp. \emph{afterminimum}) of $\sigma$ if it is at the same time a left-to-right maximum (resp. right-to-left minimum). Denote the number of foremaxima (resp. afterminima) of $\sigma$ by $\fmax\:\sigma$ (resp. $\amin\, \sigma$).    
  Note that  for the peak number $\peak'$ in \eqref{eq:peak-poly} we have  
 following identities~:
    \begin{align}\label{peak-val}
    \peak'=\val=\peak-1 \quad \text{and}\quad  \lvalley=\lpeak.
    \end{align}

        Now we recall two bijections $\Phi$ and $\Psi$ due to
         Clarke et al.~\cite{CSZ97} and Shin-Zeng~\cite{SZ12}, respectively.        
   \subsection{The bijection $\Phi$} \label{def-Phi}  
        Let  $\sigma=\sigma(1)\ldots \sigma(n)\in \S_n$,
        an {\em inversion top number} (resp. {\em inversion bottom number}) of a letter $x:=\sigma(i)$
        in the word $\sigma$ is the number of occurrences of inversions of form $(i,j)$ (resp $(j,i)$).
        A letter  $\sigma(i)$ is a {\em descent top} (resp. {\em descent bottom}) 
        if $\sigma(i)>\sigma(i+1)$ (resp. $\sigma(i-1)>\sigma(i)$).
        Given a permutation $\sigma$, we first construct two biwords, $f \choose f'$ and $g \choose g'$,  where 
        $f$ (resp. $g$) is the subword of descent bottoms (resp. nondescent bottoms) in $\sigma$ ordered increasingly, 
        and $f'$ (resp. $g'$)
        is the permutation of descent tops (resp. nondescent tops)  in $\sigma$ such that the inversion bottom (resp. top) number of each 
         letter $x:=\sigma(i)$ in $f'$ (resp. $g'$) is  $(\bca)_x\sigma$, and then form the biword $w=\left({f \atop f'}~{g \atop g'}\right)$ by concatenating $f$ and $g$, and $f'$ and $g'$, respectively.
Rearranging the columns of $w$, so that the bottom row is in increasing order, we obtain the permutation $\tau=\Phi(\sigma)$ as the top row of the rearranged bi-word.
      
  \medskip            
The following result can be found in \cite[Theorem 12]{SZ12} and its proof.
         
       \begin{lem}[Shin-Zeng ]\label{97csz}
              For $\sigma\in \S_{n}$, we have
              \begin{subequations}
       	\begin{align}\label{eq:key}
       	&(\bca, \cab, \des,\asc, \lda-\fmax, \ldd, \lvalley, \lpeak,\fmax)\sigma\nonumber\\
	=&
        ( \nest, \icr, \drop, \exc+\fix, \cda, \cdd, \cvalley,\cpeak,\fix)\Phi(\sigma)\\
     =&
       ( \nest, \cros, \exc, \drop+\fix, \cdd, \cda, \cvalley,\cpeak,\fix)(\Phi(\sigma))^{-1},\nonumber
       	\end{align}
       	\begin{align}\label{eq:keycsz}
       	(\Lval, \Lpeak, \Lda, \Ldd)\sigma=
       	( \Cval, \Cpeak, \Cda\cup\Fix, \Cdd)\Phi(\sigma),
       	\end{align}	
       	and 
       	\begin{equation}\label{231nest}
       	(\bca)_i\sigma = \nest_i\Phi(\sigma) \quad \forall i=1,\ldots,n.
       	\end{equation}
       	\end{subequations}
              \end{lem}
   \subsection{The bijection  $\Psi$}   \label{def-Psi}
        Given a permutation $\sigma \in \S_n$, let 
 \begin{equation} \label{eq:hat}
\hat{\sigma}=\begin{pmatrix}
1& 2&\ldots& n&n+1\\
\sigma(1)+1&\sigma(2)+1&\dots&\sigma(n)+1&1
\end{pmatrix},
\end{equation}
         and $\tau:=\Phi(\hat{\sigma}) \in \S_{n+1}$.
        Since the last element of $\hat{\sigma}$ is $1$, 
 the first element  of $\tau$ should be $n+1$. Define the bijection $\Psi: \S_n\to \S_n$ by 
        \begin{align}\label{def:Psi}
        \Psi(\sigma):=\tau(2)\dots\tau(n+1) \in \S_{n}.
        \end{align}

        	\begin{example}
        		If $\sigma=4~1~ 2~ 7~ 9~6~5~8~3$, then $\hat{\sigma}=5~2~3~8~10~7~6~9~4~1$, and reading from left to right,  we obtain the corresponding 
numbers 
		$(\bca)_i\hat{\sigma}:1,1,1,2,0,1,1,0,0,0$  for $i=5, 2,\ldots, 1$, and 
        		$$
        		{f\choose f'}
        		= \left(
        		{1 \atop 4}~
        		{2 \atop 9}~
        		{4 \atop 5}~
        		{6 \atop 7}~
        		{7 \atop 10}
        		\right),
        		\quad
        		{g\choose g'}
        		= \left(
        		{3 \atop 2}~
        		{5 \atop 3}~
        		{8 \atop 8}~
        		{9 \atop 6}~
        		{10 \atop 1}
        		\right).
        		$$
        		Hence
        		$$
        	w
        		= \left( {f \atop f'}~{g \atop g'} \right)
        		= \left(
        		{1 \atop 4}~
        		{2 \atop 9}~
        		{4 \atop 5}~
        		{6 \atop 7}~
        		{7 \atop 10}~
        		{3 \atop 2}~
        		{5 \atop 3}~
        		{8 \atop 8}~
        		{9 \atop 6}~
        		{10 \atop 1}
        		\right)
        		\to
        		\left(
        		{10 \atop 1}~
        		{3 \atop 2}~
        		{5 \atop 3}~
        		{1 \atop 4}~
        		{4 \atop 5}~
        		{9 \atop 6}~
        		{6 \atop 7}~
        		{8 \atop 8}~
        		{2 \atop 9}~
        		{7 \atop 10}
        		\right).
        		$$
        		Thus $\tau=\Phi(\hat{\sigma})= 10~3~5~1~4~9~6~8~2~7$, and $\Psi(\sigma)=\tau(2)\dots\tau(10)=3~5~1~4~9~6~8~2~7$.
        	\end{example}

  \begin{lem}\label{lem-thm2.9}
  For $i\in [n]$, we have 
  $$
	(\bca)_{i+1}\hat{\sigma}=\begin{cases}
	(\bac)_i\sigma+1& 
	\text{if $i+1\in \Lval\,\hat{\sigma}\cup \Lda\,\hat{\sigma}$},\\
	(\bac)_i\sigma& \text{if $i+1\in \Lpeak\,\hat{\sigma}\cup \Ldd\,\hat{\sigma}$}.
	\end{cases}
	$$ 
  \end{lem}
  \begin{proof}
  An increasing (resp. decreasing) run of $\sigma$ is 
   a maximum consecutive increasing (resp. decreasing) subsequence 
   $R:=\sigma(i)\sigma(i+1)\ldots \sigma(j)$ of $\sigma$ such that 
   $\sigma(i-1)>\sigma(i)$ and $\sigma(j)>\sigma(j+1)$ (resp.
     $\sigma(i-1)<\sigma(i)$ and $\sigma(j)<\sigma(j+1)$) with $1\leq i\leq j\leq n$.
        	For any $i\in [n]$, as $\hat\sigma(n+1)=1$,
	there is a unique way to 
	write 
\[
	\hat\sigma=\begin{cases}
	w_1(i+1) u_1d_2\ldots u_{k-1}d_k& 
	\text{if $i+1\in \Lval\hat{\sigma}\cup \Lda\hat{\sigma}$},\\
	w_1(i+1) d_1u_2d_2\ldots u_kd_k& \text{if $i+1\in \Lpeak\hat{\sigma}\cup \Ldd\hat{\sigma}$},
	\end{cases}
\]
	where $u_i$ (resp. $d_i$) is an increasing (resp. decreasing) run, 
	and $(i+1)u_1$ (resp. $(i+1)d_1$) is an increasing (resp. decreasing) sequence. We say that a run $R$ covers $i$ if $i$ is bounded by  $\max(R)$
	 and  $\min(R)$.  It is not hard to show that
	 \[
\#\{j\geq 2: \text{$u_j$ covers $i+1$}\}
=\begin{cases}
	\#\{j\geq 2: \text{$d_j$ covers $i+1$}\}+1& 
	\text{if $i+1\in \Lval\hat{\sigma}\cup \Lda\hat{\sigma}$},
	\\
	\#\{j\geq 2: \text{$d_j$ covers $i+1$}\}& 
	\text{if $i+1\in \Lpeak\hat{\sigma}\cup \Ldd\hat{\sigma}$}.
	\end{cases}
\] 
	Since  $(\bac)_i$ (resp.  $(\bca)_i$)
	is the number of increasing (resp. decreasing)
	runs covering $i$ to the right of $i$,  we are done.
 \end{proof}

We use the aforementioned statistics to define  variant boundary conditions. Given a permutation $\sigma\in\S_n$ with convention $\infty-0$, 
 the  number of corresponding peaks, valleys, double ascents, and double descents of permutation $\sigma\in \S_n$ is denoted by $\rpeak\, \sigma$, $\rvalley\, \sigma$, $\rda\,\sigma$, $\rdd\, \sigma$ respectively.
A double descent $\sigma(i)$ is called a \emph{aftermaximum} (resp. foreminimum) of $\sigma$ if it is at the same time a right-to-left maximum (resp. left-to-right minimum). Denote the number of aftermaxima (resp. foreminimum) of $\sigma$ by $\amax\,\sigma$ (resp. $\fmin\,\sigma$).  
For $\sigma=\sigma(1)\sigma(2)\cdots\sigma(n)                                                                                                                                                                                           \in\S_{n}$, we define two  permutations 
$\sigma^c$ and $\sigma^r$ by 
\begin{align} \label{def:rc}
\sigma^c(i)=n+1-\sigma(i)\quad \textrm{and}\quad 
 \sigma^r(i)=\sigma(n+1-i)\quad \text{ for}\quad  i\in [n].
 \end{align}
It is not difficult to verify the following properties
\begin{subequations}
\begin{align}
&(\bca, \cab, \des,\lda-\fmax,\ldd,\lvalley,\fmax)\:\sigma\label{patternres1}\\
= &(\acb, \bac, \asc, \rdd-\amax, \rda, \rvalley, \amax)\:\sigma^r\label{patternres2}\\
=&(\cab, \bca, \des, \lda-\amin, \ldd, \lpeak, \amin)\,\sigma^{r\circ c}\label{patternres3}\\
=&(\bac, \acb, \asc, \rdd-\fmin, \rda, \rvalley, \fmin)\,\sigma^{r\circ {c}\circ r},\label{patternres4}
\end{align}
\end{subequations}
where $\sigma^{r\circ c}=(\sigma^r)^c$ and $\sigma^{r\circ {c}\circ  r}=(\sigma^{r\circ c})^r=(\sigma^r)^{c\circ r}$.

\subsection{The star variation}
For $\sigma=\sigma(1)\cdots\sigma(n) \in \S_n$,  we define
its \emph{star companion}
$\sigma^*$ as a permutation of $\{0, \ldots, n\}$ by
\begin{equation}
\label{eq:star}
\sigma^*=\begin{pmatrix}
0& 1&2&\ldots& n\\
n& \sigma(1)-1&\sigma(2)-1&\dots&\sigma(n)-1
\end{pmatrix}.
\end{equation}          
       We define the following sets of cyclic star statistics for $\sigma$: 
       \begin{subequations}
       \begin{align}
       \Cpeak^*\sigma &=   
       \{i \in [n-1]:  (\sigma^{*})^{-1}(i)<i>\sigma^{*}(i)\},\label{Eq.3.3}\\
       \Cval^{*}\sigma &=   \{i \in [n-1]:  (\sigma^{*})^{-1}(i)>i<\sigma^{*}(i)\},\\
       \Cda^{*}\sigma &=   \{i \in [n-1]: (\sigma^{*})^{-1}(i)<i<\sigma^{*}(i)\},\label{Eq.3.8}\\
       \Cdd^{*}\sigma &=   \{i \in [n-1]: (\sigma^{*})^{-1}(i)>i>\sigma^{*}(i)\},\label{Eq.3.7}\\
       \Fix^{*}\sigma &=  \{i\in [n-1]: i=\sigma^{*}(i)\},\label{Eq.3.4}\\
       \Wex^{*}\sigma &= \{i\in [n-1]: i\le\sigma^{*}(i)\},\\
       \Drop^{*}\sigma &=  \{i\in [n]: i>\sigma^{*}(i)\}.\label{Eq.3.9}
 \end{align}
 \end{subequations}
The corresponding cardinalties are denoted by 
$\cpeak^*$, $\cval^*$, $\cda^*$,  $\cdd^*$,  $\fix^*$, $\wex^*$ and 
$\drop^*$, respectively.
By \eqref{Eq.3.3}, \eqref{Eq.3.7} and \eqref{Eq.3.9},
 we have $\drop^*-1=\cdd^*+\cpeak^*.$
Let $\cyc\, \sigma$ be the number of cycles 
of $\sigma$ and $\cyc^*\sigma:=\cyc\sigma^*$.
For example, for $\sigma=3762154$, we have $\sigma^*=72651043$, which has two cycles $1\to 2\to 6 \to 4\to 1$ and  $7 \to 3 \to 5 \to 0 \to 7$.  Thus 
$\cyc^*\,\sigma=2$.

%
        
       
%
%

  For any subsebt $S\subset \N$, we write $S+1:=\{s+1:s\in S\}$.

\begin{thm}\label{SZsetprop}
          For $\sigma\in\S_n$, we have 
          	\begin{align}\label{set-stat}
          	(\Val,\Peak_n,\Da,\Dd)\sigma=
	(\Cval^*,\Cpeak^*,\Cda^*\cup\Fix^*,\Cdd^*)\Psi(\sigma)
          \end{align}
  with $\Peak_n\sigma=\Peak\sigma\setminus \{n\}$ and 
    \begin{align}\label{nest-cros-i}
  ((\bac)_i,\;(\cab)_i)\sigma=  (\nest_i, \; \cros_i)\Psi(\sigma)
          		\quad\text{for}\quad  i\in [n].
		\end{align}
          \end{thm}

         \begin{proof}   	   For $\sigma\in\S_n$,
         by definition~\eqref{eq:hat}, \eqref{def:Psi} and \eqref{eq:star}, 
we have 
\begin{equation}\label{star+1}
\tau(i+1)=
(\Psi(\sigma))^*(i)+1.
\end{equation} 
and
        \begin{align}
       ((\Val\,\sigma+1)\cup \{1\}, 	\Peak\,\sigma+1, \Da\,\sigma+1, 	\Dd\,\sigma+1)= (\Lval, \Lpeak, \Lda,\Ldd)\hat \sigma.\label{VAL}
    \end{align}
        	For $2\leq i\leq n$,  by \eqref{def:cval} and \eqref{star+1} we have the following equivalences:
	\begin{align*}
i\in \Cval\tau	\Longleftrightarrow 	i-1\in\Cval^*(\Psi(\sigma)) 
\end{align*}
and
        \begin{align*}
         i<\tau(i) \;\text{and} \;i<\tau^{-1}(i)
        	\Longleftrightarrow i< \Psi(\sigma)(i-1)\; \text{and}\; i-1<(\Psi(\sigma))^{-1}(i).
       \end{align*}
        	 Thus, by \eqref{eq:hat} and \eqref{def:Psi},
            \begin{align}\label{cval1}
            (\Cval^*\Psi(\sigma)+1)\cup \{1\}=\Cval\Phi(\hat{\sigma}).
            \end{align}
  In the same vein, we have 
        	\begin{align}
	\Cpeak^*\Psi(\sigma)+1&=\Cpeak\Phi(\hat{\sigma})\setminus \{n+1\}, \nonumber\\
	\Cda^*\Psi(\sigma)\cup \Fix^*\Psi(\sigma)+1&=\Cda\Phi(\hat{\sigma})\cup\Fix\Phi(\hat{\sigma}),\label{cdd1}\\
\Cdd^*\Psi(\sigma)+1&=\Cdd\Phi(\hat{\sigma}).\nonumber
\end{align}
Comparing \eqref{VAL} and \eqref{cval1}-\eqref{cdd1} and using 
\eqref{eq:keycsz} we derive  \eqref{set-stat}.
  
  Next, for \eqref{nest-cros-i}, we only  prove $\nest_i\Psi(\sigma)=(\bac)_i\sigma$ and leave 
$\cros_i\Psi(\sigma)=(\cab)_i\sigma$  to the interested reader. 
 By Lemma~\ref{lem-thm2.9}  we have 
 \begin{align}\label{def:hatbis}
	(\bca)_{i+1}\hat{\sigma}=\begin{cases}
	(\bac)_i\sigma+1& 
	\text{if $i+1\in \Lval\,\hat{\sigma}\cup \Lda\,\hat{\sigma}$},\\
	(\bac)_i\sigma& \text{if $i+1\in \Lpeak\,\hat{\sigma}\cup \Ldd\,\hat{\sigma}$}.
	\end{cases}
\end{align}
  So,  if we  show that 
 \begin{align}\label{nestnest}
 (\nest)_{i+1}\tau=\begin{cases}
 (\nest)_i\Psi(\sigma)+1& 
 \text{if $i+1\in \Cval\,\tau\cup \Cda\,\tau \cup \Fix\,\tau$},\\
 (\nest)_i\Psi(\sigma)& \text{if $i+1\in \Cpeak\,\tau\cup \Cdd\,\tau$},
 \end{cases}
 \end{align}
as  $\tau=\Phi(\hat\sigma)$ and  by \eqref{eq:keycsz}, 
\begin{align}\label{eq:keycszbis}
       	(\Lval, \Lpeak, \Lda, \Ldd)\hat \sigma=
       	( \Cval, \Cpeak, \Cda\cup\Fix, \Cdd)\Phi(\hat \sigma),
       	\end{align}
	the result follows from the identity  	
 $\nest_{i}\tau=(\bca)_{i}\hat{\sigma}$  (see \eqref{231nest}).
 
Now we prove \eqref{nestnest}. By \eqref{def:nest}
the index $\nest_i\:\sigma$ ($i\in [n]$) can 
be characterized in terms of $\sigma^*$ (see \eqref{eq:star}) as follows: 
\begin{align}
\nest_i\:\sigma =\#\{j\in[n]:\:j<i\leq\sigma^*(i)<\sigma^*(j)\;\textrm{or}
\;\sigma^*(j)<\sigma^*(i)< i<j\}\label{nest**}.
\end{align}
We consider  three cases of $i+1$.
        		\begin{itemize}
        		\item [(a)]	 if $i+1<\tau(i+1)$, then $i<(\Psi(\sigma))^*(i)$.
        		By \eqref{star+1}, 
        		we have 
        		\begin{subequations}
        		\begin{align}
        		&\#\{j\in[n]:j+1<i+1<\tau(i+1)<\tau(j+1)\}\label{eq:1}\\
		=&\#\{j\in[n]:j<i\leq(\Psi(\sigma))^*(i)<(\Psi(\sigma))^*(j)\}\nonumber
		\end{align}
		and
        		\begin{align}
		&\#\{j\in[n]: j+1>i+1\geq\tau(i+1)>\tau(j+1)\}\label{eq:2}\\
        		&=\#\{j\in[n]: j>i>(\Psi(\sigma))^*(i)>(\Psi(\sigma))^*(j)\}.\nonumber
        		\end{align}
        		\end{subequations}
  Since $\tau(1)=n+1$ and $1<i+1<\tau(i+1)<n+1$,
        		by \eqref{def:nest}, \eqref{eq:1} and \eqref{eq:2} we obtain
      	\begin{align*}
        	\nest_{i+1}\:\tau 
            =1+&\#\{j\in[n]:j+1<i+1<\tau(i+1)<\tau(j+1)\}\\
             +&\#\{j\in[n]: j+1>i+1\geq\tau(i+1)>\tau(j+1)\}\\
             =1+&\#\{j\in[n]:j<i\leq(\Psi(\sigma))^*(i)<(\Psi(\sigma))^*(j)\}\\
             +&\#\{j\in[n]: j>i>(\Psi(\sigma))^*(i)>(\Psi(\sigma))^*(j)\},
            	\end{align*} 
  	which is equal to $\nest_i\Psi(\sigma)+1$ by \eqref{nest**}.
   	\item [(b)]  if $i+1>\tau(i+1)$, then 
   	$i>(\Psi(\sigma))^*(i)$.
   	Similarly to (a)  we get 
   	$
   	\nest_{i+1}\tau=\nest_i\Psi(\sigma)$.      
   	\item [(c)]  if $i+1=\tau(i+1)$,
	then $i=(\Psi(\sigma))^*(i)$.
    It is easy to see that 
    \begin{equation}\label{final}
    \#\{j\in[n+1]: j>i+1>\tau(j)\}=\#\{j\in[n+1]: j<i+1<\tau(j)\}.
    \end{equation}
As $\tau(1)=n+1$, we have 
	\begin{align*}
	\#\{j\in[n+1]: j>i+1>\tau(j)\}
	=&\#\{j\in[n+1]: j<i+1<\tau(j)\} \\
		=&\#\{j\in[n]: j+1<i+1<\tau(j+1)\}+1\\
	=&\#\{j\in[n]:j<i<(\Psi(\sigma))^*(j)\}+1 \quad (\textrm{by}\; \eqref{star+1})
	\end{align*}
	 Then,
	  we have  $\nest_{i+1}\tau=\nest_i\Psi(\sigma)+1$
	 by using \eqref{def:nest} (resp. $\eqref{nest**}$ ) to compute $\nest_{i+1}\tau$ (resp.  $\nest_i\Psi(\sigma)$).
     \end{itemize}
        \end{proof}
    Since $\asc=\val+\da$, $\des=\peak+\dd-1$, $\wex^*=\cval^*+\cda^*+\fix^*$, $\drop^*-1=\cdd^*+\cpk^*$,
    we get the following result in \cite[Theorem 12]{SZ12}. 
    \begin{cor}[Shin-Zeng]\label{SZbijection} For $\sigma\in \S_n$ we have 
    	\begin{align}
    	&(\bac, \cab, \des, \asc, \da, \dd, \valley)\sigma\nonumber\\
    	=&(\nest, \cros,\drop^*-1,\wex^*,\cda^*+\fix^*,\cdd^*,\cvalley^*)\Psi(\sigma).\label{starCSZ97}
    	\end{align}
    \end{cor}

        \subsection{Laguerre histories as permutation encodings}
        A 2-\emph{Motzkin path} is a lattice path starting and ending on the horizontal axis but never going below it, with possible steps $(1,1)$, $(1,0)$, and $(1,-1)$, where the level steps $(1,0)$ can be given either of two colors: 
        \emph{blue} and \emph{red}, say. The \emph{length} of the path is defined to be the number of its steps. For our purpose 
        it is convenient to identify 
        a 2-\emph{Motzkin path}  of length $n$  as a word $\mathbf{s}:=s_1\ldots s_n$ on the alphabet $\{\su, \sd, \la, \lb\}$ 
		such that  $|s_1\ldots s_n|_\su= [s_1\ldots s_n|_\sd$ and the \emph{height} of the $i$th step is nonnegative, i.e., 
		\begin{equation}\label{heigt-condition}
		h_{i}(\mathbf{s}):=|s_1\ldots s_i|_\su-[s_1\ldots s_i|_\sd\geq 0 \quad (i=1, \ldots, n),
		\end{equation}
	 where $|s_1\ldots s_i|_\su$ is the number of letters $\su$ in the word $s_1\ldots s_i$. By \eqref{2-motzkin} we see that  the number of 2-Motzkin paths of length $n$ is the Catalan number $C_{n+1}$.

		A \emph{Laguerre history} (resp. \emph{restricted Laguerre history}) of length  $n$ is  a pair $(\mathbf{s, p})$, where $\mathbf{s}$ is 
		a 2-Motzkin path $s_1\ldots s_n$ 
		and $\mathbf{p}=(p_1, \ldots, p_n)$ with $0\leq p_i\leq h_{i-1}(\mathbf{s})$
		 (resp. $0\leq p_i\leq h_{i-1}(\mathbf{s})-1$ if $s_i=\lr$ or $\sd$) with $h_0(\mathbf{s})=0$.
		Let $\LH_n$ (resp. $\LH^*_n$) be the set of Laguerre histories 
		(resp. restricted Laguerre histories)
		of length $n$.
		There are several  well-known  bijections 
		between $\S_{n}$ and $\LH^*_{n}$ and 
		$\LH_{n-1}$, see \cite{DV94, CSZ97}.

	\medskip	
	\begin{figure}[t]
 	\begin{center}
 		\begin{tikzpicture}[scale=0.9]
 		\draw[step=1cm, gray, thick, dotted] (0, 0) grid (8,2);
 		
 		\draw[black] (0, 0)--(1, 1)--(2, 1)--(3, 1)--(4, 0)--(5, 1)--(6, 1)--(7, 1)--(8, 0);
 		\draw[black] (0,0) node {$\bullet$};
 		\draw[black] (1,1) node {$\bullet$};
 		\draw[black] (2,1) node {$\bullet$};
 		\draw[black] (3,1) node {$\bullet$};
 		\draw[black] (4,0) node {$\bullet$};
 		\draw[black] (5,1) node {$\bullet$};
 		\draw[black] (6,1) node {$\bullet$};
 		\draw[black] (7,1) node {$\bullet$};
 		\draw[black] (8,0) node {$\bullet$};

 		\draw [-,thin, blue](1,1)--(2,1);
 		\draw [-,thin, red](2,1)--(3,1);
 		\draw [-,thin, red](5,1)--(6,1);
 		\draw [-,thin, blue](6,1)--(7,1);
 	
 		\tiny{
 			\draw[black] (-0.5,-0.5) node {$p_i$};
 			\draw[black] (0.5,-0.5) node {$0$};
 			\draw[black] (1.5,-0.5) node {$0$};
 			\draw[black] (2.5,-0.5) node {$0$};
 			\draw[black] (3.5,-0.5) node {$1$};
 			\draw[black] (4.5,-0.5) node {$0$};
 			\draw[black] (5.5,-0.5) node {$1$};
 			\draw[black] (6.5,-0.5) node {$1$};
 			\draw[black] (7.5,-0.5) node {$0$};

 			\draw[black] (1.5,1.5) node {$\lb$};
 			\draw[black] (2.5,1.5) node {$\lr$};
 			\draw[black] (5.5,1.5) node {$\lr$};
 			\draw[black] (6.5,1.5) node {$\lb$};
 		
 		}	
 		\end{tikzpicture}
 		\caption{\label{lhexample} A Laguerre history $(\mathbf{s, p})$ of lenth 8. }
 	\end{center}
 \end{figure}
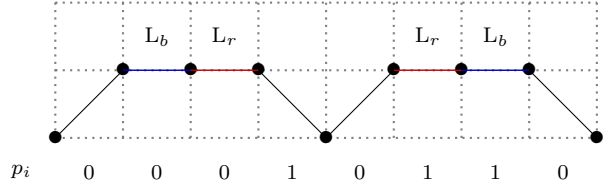

			\subsection{Fran\c con-Viennot bijection}
		We recall a version of Fran\c con and Viennot's bijection 
		$\psi_{FV}: \S_{n+1}\to \LH_{n}$. Given $\sigma\in \S_{n+1}$,  the Laguerre history $\psi_{FV}(\sigma)=(\mathbf{s, p})$ is defined as follows: 
		\begin{align}\label{FV}
		s_i=\begin{cases}
		\su &\textrm{if} \; i\in \Val\sigma
		\\
		\sd &\textrm{if} \; i\in \Peak\sigma
		\\
		\lb &\textrm{if} \; i\in \Da\sigma
		\\
		\la &\textrm{if} \; i\in \Dd\sigma
		\end{cases}
		\end{align}
		and  $p_i=(\bac)_i~\sigma$ for $i=1, \ldots, n$.
		
		For example, if  $\sigma=4~1~ 2~ 7~ 9~6~5~8~3\in \S_9$, then
		$$(\mathbf{s, p})=((\su,\lb,\la,\sd,\su,\la,\lb,\sd),(0,0,0,1,0,1,1,0))\in \LH_{8},$$
		which is depicted in Figure~\ref{lhexample}.
%
%

For  $\sigma\in \S_{n+1}$, 
	we define the following sets
	\begin{subequations}
\begin{align}
	\Scval&= \{i\in [n]: i<\sigma(i)\; \text{and}\; i+1\leq\sigma^{-1}(i+1)\},\label{eq:shift1}\\
	\Scpeak&=\{i\in [n]: i\geq\sigma(i)\; \text{and}\; i+1>\sigma^{-1}(i+1)\},
	\label{eq:shift2}\\
	\Sde&= \{i\in [n]: i<\sigma(i)\; \text{and}\;i+1>\sigma^{-1}(i+1)\},
	\label{eq:shift3}\\
	\Sdn&=\{i\in [n]: i\geq\sigma(i)\; \text{and}\; i+1\leq\sigma^{-1}(i+1)\}.\label{eq:shift4}
\end{align}
\end{subequations}
	
	Yan, Zhou and Lin~\cite{YZL19} constructed a  
bijection $\psi_{{YZL}}:\S_{n+1}\to \L_n$, which can be defined as follows.
For $\sigma\in \S_{n+1}$,      let $\psi_{YZL}(\sigma)=(\mathbf{s, p})$ with 
\begin{align}\label{YZLLH}
s_i=\begin{cases}
\su &\textrm{if} \; i\in \Scval\sigma,
\\
\sd &\textrm{if} \; i\in \Scpeak\sigma,
\\
\lb &\textrm{if} \; i\in \Sde\sigma,
\\
\la &\textrm{if} \; i\in \Sdn\sigma,
\end{cases}
\end{align}
and  $p_i=\nest_i\,\sigma$ for $i=1, \ldots, n$.

    \begin{thm}\label{HMZ}
    We have $\psi_{FV}=\psi_{{YZL}}\circ\Psi$. 
\end{thm} 

\begin{proof}

Let	$\psi=\psi_{{FV}}\circ\Psi^{-1}$, which 
 is a bijection from $\S_{n+1}$ to $\L_n$.
By Theorem~\ref{SZsetprop}, for $\sigma\in \S_{n+1}$, we can define $\psi(\sigma)=(\mathbf{s, p})$ as follows: 
	 for $i=1, \ldots, n$,
	\begin{align}\label{HMZ-STAR}
	s_i=\begin{cases}
	\su &\textrm{if} \; i\in \Cval^*\sigma,
	\\
	\sd &\textrm{if} \; i\in \Cpeak^*\sigma,
	\\
	\lb &\textrm{if} \; i\in \Cda^*\sigma\cup\Fix^*\sigma, 
	\\
	\lr &\textrm{if} \; i\in \Cdd^*\sigma,
	\end{cases}
	\end{align}
	with $p_i=\nest_i\sigma$. Comparing \eqref{YZLLH} and \eqref{HMZ-STAR} 
it suffices   to show that
for $\sigma\in \S_{n+1}$,
	\begin{align}\label{star-shift}
	(\Scval, \;\Scpeak,\;  \Sdn, \;\Sde)\sigma
	=(\Cval^*,\; \Cpeak^*,\;\Cdd^*,\;\Cda^*\cup\Fix^*)\sigma.
	\end{align}
We just prove $\Scval\sigma=\Cval^*\sigma$ and omit the similar proof of other cases. 
As $\Cval^*(\sigma)=\{i:i+1<\sigma(i), \; i< \sigma^{-1}(i+1)\}$, comparing with \eqref{eq:shift1}
we need only to  show that $\Scval\sigma \subset \Cval^*\sigma$.
If  $i\in \Scval(\sigma)$, then 
$i<\sigma(i)$ and $ i+1\leq \sigma^{-1}(i+1)$. Suppose 
 $i+1=\sigma(i)$,  then $\sigma^{-1}(i+1)=i$, which contradicts the second inequality.
So  $i+1<\sigma(i)$, and $i\in \Cval^*\sigma$. 
 We are done.  
\end{proof}

\begin{figure}[t]
	{\small
		\begin{tabular}{c|c|c|c|c|c|c|c}
			{$\sigma\in\DD_{4,k}$}&{$\sigma$} 
			&{$\sigma^r$} &{$\Psi(\sigma)\in\SDE_{4,k}$ } &
			$(\cab)\sigma$ &$(\bac)\sigma$
			&$\inv\Psi(\sigma)$ & $\exc\Psi(\sigma)$\\
			\hline
			$k=0$ &1324 & 4231 & 1423 & 
			0 & 1 & 2 & 1  
			\\
			\hline
			 &1423 & 3241 & 1432 &   1 & 0 & 3 & 1 \\
			 &2314 & 4132 & 4123 &   0 & 2 & 3 & 1 \\
		    &2413 & 3142 & 4132 &   1 & 1 & 4 & 1\\
		    $k=1$&3412 & 2143 & 3214 & 1 & 0 & 3 & 1 \\
		    &2134 & 4312 & 3124 & 0 & 1 & 2 & 1 \\
	        &3124 & 4213 & 4213  & 1 & 1 & 4 & 1\\
            &4123 & 3214 & 4231   & 2 & 0 & 5 & 1\\	
		\end{tabular}
	}
	\caption{Illustration of $\Psi$ on $\DD_{4,k}$ with their statistics.}
	\label{fig:table2}
\end{figure}

  Given a 2-Motzkin path $\mathbf{s}$ of length $n$ 
we define $\theta(\mathbf{s})$ to be the 
 2-Motzkin path obtained by switching all the letters 
$\lb$ with  $\lr$  in $\mathbf{s}$.
By abuse of notation, for   a Laguerre history $(\mathbf{s, p})\in \LH_n$  we define 
\begin{align}\label{def:theta}
\theta(\mathbf{s, p})=(\theta(\mathbf{s}), \mathbf{p}).
\end{align}

   \begin{cor}\label{starpqzhuang}
The two sextuple  statistics 
   $$
   (\nest, \cros, \exc, \cdd^*,\cda^*+\fix^*, \cpeak^*)\quad \textrm{and}\quad
   (\bac, \cab, \des,  \da, \dd, \peak-1)
   $$ 
    		are equidistributed on $\S_n$.
    	\end{cor}
   
 \begin{proof}
 For $\sigma\in\S_n$, let $\tau=\psi^{-1}\circ\theta\circ\psi_{{FV}}(\sigma)$.
 It follows from \eqref{def:theta}, \eqref{FV} and \eqref{HMZ-STAR} that
 \begin{align*}
 &(\Val,\; \Peak_n,\; \Dd,\; \Da,\; (\bac)_i)\sigma\\
  =&(\Cval^*,\;\Cpeak^*,\;  \Cda^*\cup\Fix^*,\; \Cdd^*,\; \nest_i)\tau, \quad \forall i\in [n].
  \end{align*}
Let $\psi_{{FV}}(\sigma)=(\mathbf{s, p})$ and 
$\psi(\tau)=(\mathbf{s', p'})$. Then $h_i(\mathbf{s, p})=h_i(\mathbf{s', p'})$ for all $i\in[n]$. 
It is not difficult to prove by induction that	
\begin{subequations}
\begin{align}
(\bac)_i\sigma+(\cab)_i\sigma&=h_{i-1}(\mathbf{s, p}),\\
\nest_i\sigma+\cros_i\sigma&=h_{i-1}(\mathbf{s', p'}).
\end{align}  
\end{subequations}
Thus we have $(\cab)_i\sigma=\cros_i\tau$. 
As $\exc=\wex^*=\cval^*+\cda^*+\fix^*$,
$\des=\val+\dd$,  $\cpeak^*=\cval^*$, and $\val=\peak-1$,
 the proof is completed.
\end{proof}

For $k\in [n]$ we define the subsets of $\S_n$:
\begin{subequations}
\begin{align}
\DD_{n,k}:=&\{\sigma\in\S_n:\des\,\sigma=k,\;\dd\,\sigma=0\},\\
\DE^*_{n,k}:=&\{\sigma\in\S_n:\exc\,\sigma=k,\;  \cda^*\,\sigma+\fix^*\,\sigma=0\},\label{DE-star}\\
\SDE_{n,k}:=&\{\sigma\in\S_n:\exc\,\sigma=k,\; 
 \sde(\sigma)=0\}.\label{SDE}
\end{align}
\end{subequations}

\begin{thm}\label{Thm:YZL-HMZ} For $0\leq k\leq (n-1)/2$ 
	we have 
	\begin{subequations}
	\begin{align}
\gamma_{n,k}(q)&:=\sum_{\sigma\in \DD_{n,k}}q^{2(\cab)\sigma+(\bac)\sigma}\\
&=\sum_{\sigma\in \DE^*_{n,k}}q^{\inv\,\sigma-\exc\,\sigma}\label{eq:YZL-HMZ2}\\
&=\sum_{\sigma\in \SDE_{n,k}}q^{\inv\,\sigma-\exc\,\sigma}\label{eq:YZL-HMZ1}.
	\end{align}	
	\end{subequations}		
\end{thm} 
\begin{proof}
		For $\sigma\in \S_n$, recall that $\sigma^r:=\sigma(n)\cdots\sigma(2)\sigma(1)$ (see \eqref{def:rc}).
	By \eqref{starCSZ97},
	\begin{align*}
	\bigl(2(\cab)+\bac\bigr)\sigma=\bigl(2(\bac)+\cab\bigr)\sigma^r=(2\nest+\cros)\Psi(\sigma^r).
	\end{align*}
	Invoking the following formula for inversion numbers (cf. \cite[Eq. (40)]{SZ10})
\begin{align}\label{INV}
\inv=\exc+2\nest+\cros,
\end{align} 
	we derive
	\begin{align}\label{inv=nest+cros}
	\bigl(2(\cab)+(\bac)\bigr)\sigma= (\inv-\exc)\Psi(\sigma^r).
	\end{align}
	Besides,
	by \eqref{starCSZ97} and \eqref{star-shift} we have  
	\begin{align*}
	(\des, \dd)\sigma&=(\asc, \da)\sigma^r\\
	&=(\wex^*, \cda^*+\fix^*)\Psi(\sigma^r)\\
	&=(\cval^*+\cda^*+\fix^*, \cda^*+\fix^*)\Psi(\sigma^r).
	\end{align*}   
	Hence, if $\dd(\sigma)=(\cda^*+\fix^*)\Psi(\sigma)=0$, 
	from \eqref{starCSZ97} we see  that 
	$\sigma\in \DD_{n,k}$ if and only if $\Psi(\sigma^r)\in \DE^*_{n,k}$.	
	By \eqref{inv=nest+cros} this implies \eqref{eq:YZL-HMZ2}. 
		 Finally, we derive \eqref{eq:YZL-HMZ1} from  \eqref{star-shift}. 
\end{proof}
	\begin{remark} Yang-Zhou-Lin~\cite{YZL19} proved that 
	$$
	\gamma_{n,k}(q)=\sum_{\sigma\in \DD_{n,k}}q^{(\cab)\sigma+2(\bac)\sigma}
	$$
	which first appeared  as the $\gamma$-coefficents of 
	the polynomial 
	$\sum_{\sigma\in \S_n} t^{\exc\sigma} q^{\inv\sigma-\exc\sigma}$ in \cite{SZ16}.
	\end{remark}

 \subsection{Restricted Fran\c con-Viennot  bijection} 
 	We recall a restricted version of Fran\c con and Viennot's bijection 
		$\phi_{{FV}}: \S_{n}\to \LH_{n}^*$. Given $\sigma\in \S_{n}$,  the Laguerre history $(\mathbf{s, p})$ is defined as follows: 
		\begin{align}\label{res-FV}
		s_i=\begin{cases}
		\su &\textrm{if} \; i\in \Lval\sigma
		\\
		\sd &\textrm{if} \; i\in \Lpeak\sigma
		\\
		\lb &\textrm{if} \; i\in \Lda\sigma
		\\
		\la &\textrm{if} \; i\in \Ldd\sigma
		\end{cases}
		\end{align}
		and  $p_i=(\bca)_i~\sigma$ for $i=1, \ldots, n$.

       \subsection{Foata-Zeilberger  bijection}
       This  bijection  $\phi_{{FZ}}$ encodes permutations 
       using  cyclic statistics.
       Given $\sigma\in \S_{n}$, 
       $\phi_{{FZ}}: \S_{n}\to \L_n^*$ is 
       for $i=1, \ldots, n$,
       \begin{align}\label{resFZ}
       s_i=\begin{cases}
       \su &\textrm{if} \; i\in \Cval\sigma
       \\
       \sd &\textrm{if} \; i\in \Cpeak\sigma
       \\
       \lb &\textrm{if} \; i\in \Cda\sigma\cup\Fix\sigma
       \\
       \la &\textrm{if} \; i\in \Cdd\sigma
       \end{cases}
       \end{align}
       with $p_i=\nest_i\sigma$.
             By \eqref{eq:keycsz} and \eqref{231nest}, we can build a comutative diagram,
       see the right diagram of Figure~\ref{FVBFZB}.

\subsection{Pattern avoidances and 2-Motzkin paths}

We shall consider the so-called \emph{vincular patterns} \cite{BS}. The number of occurrences of vincular patterns $\cab$, $\bca$, $\bac$ and $\acb$ in $\pi\in \S_n$ are defined (cf. \eqref{31-2}) by
\begin{equation}\label{def:vincular}
\begin{array}{ll}
(\cab)\:\pi&=\#\{(i,j):  i+1<j\le n \text{ and } \pi(i+1)<\pi(j)<\pi(i)\},\\
(\bca)\:\pi&=\#\{(i,j):  j<i<n \text{ and } \pi(i+1)<\pi(j)<\pi(i)\},\\
(\bac)\:\pi&=\#\{(i,j): j<i<n \text{ and } \pi(i)<\pi(j)<\pi(i+1) \},\\
(\acb)\:\pi&=\#\{(i,j): i+1<j\le n \text{ and } \pi(i)<\pi(j)<\pi(i+1) \}.
\end{array}
\end{equation}
Similarly, we use $\S_n(\cab)$ to denote the set of permutations of length $n$ that avoid the vincular pattern $\cab$, etc. 
In order to apply  Laguerre history 
to  count pattern-avoiding permutations,  we will need the following results in
\cite[Lemma 2.8 and 2.9]{FTHZ19}.
\begin{lem}\label{general-231-321}\cite[Lemma 2.8]{FTHZ19}
For any $n\geq1$, we have
\begin{align}
\S_n(\bac)&=\S_n(213), \quad
\S_n(\cab)=\S_n(312),\label{eq1} \\
\S_n(\acb)&=\S_n(132), \quad
\S_n(\bca)=\S_n(231). \label{eq4}
\end{align}
\end{lem}

\begin{lem}\cite[Lemma~2.9]{FTHZ19}\label{321:nest}
\begin{itemize}
\item[(i)] A permutation $\pi\in \S_n$ belongs to $\S_n(321)$ if and only if $\nest\pi=0$. 
\item[(ii)] The mapping $\varPhi$ has the property that  $\varPhi(\S_{n}(231))=\S_n(321)$. 
\end{itemize}
\end{lem}

We use  $\2MP_n$ to denote  the set of 2-Motzkin paths
of length $n$ and $\2MP^{*}_n$ to denote its subset that is composed 
of 2-Motzkin paths without $\la$-step at level zero, i.e., 
if $h_{i-1}=0$, then $s_i\neq \la$.
Let $\widetilde{\phi}_{FV}$, $\widetilde{\phi}_{FZ}$, $\widetilde{\psi}_{FV}$  and 
$\widetilde{\psi}_{YZL}$ be the restriction of $\phi_{FV}$, $\phi_{FZ}$, $\psi_{FV}$  and 
$\psi_{YZL}$ on the sets  $\S_n(231)$, $\S_n(321)$, $\S_{n+1}(213)$ and $\S_{n+1}(321)$, respectively.

\begin{thm}\label{thm2.7} We have 
\begin{enumerate}
\item The mapping $\widetilde{\phi}_{FV}$ is a bijection from $\S_n(231)$ to $\2MP_n^*$.
\item The mapping $\widetilde{\phi}_{FZ}$ is a bijection from $\S_n(321)$ to $\2MP_n^*$.
\item The mapping $\widetilde{\psi}_{FV}$ is a bijection from $\S_n(213)$ to $\2MP_n$.
\item The mapping $\widetilde{\psi}_{YZL}$ is a bijection from $\S_n(321)$ to $\2MP_n$.
\end{enumerate}
\end{thm}

\begin{proof}
We just prove (1) and leave the others to the reader.
If  $\sigma_1$, $\sigma_2\in\S_n(231)$,  let 
$\phi_{FV}(\sigma_i)=(\mathbf{s_i}, \mathbf{p_i})$ for $i=1, 2$. 
By definition we have 
$(\bca)\sigma_1=(\bca)\sigma_2=0$, which implies that 
$\mathbf{p_1}=\mathbf{p_2}=(0, 0, \cdots, 0)$;
as $\phi_{FV}$ is a bijection, we derive that 
$\mathbf{s}_1\neq \mathbf{s}_2$. Hence,
the mapping $\widetilde{\phi}_{FV}$ is an injection from $\S_{n}(231)$ to 
$\2MP^{*}_n$.  Noticing that the g.f. $\sum_{n\geq 0} |\2MP^{*}_n|z^n$ has the continued fraction expansion \eqref{2-motzkin-star} with $t=1$,  we derive that $|\S_{n}(231)|=|\2MP^{*}_n|=C_n$.   Thus, the mapping $\widetilde{\phi}_{FV}$ is a bijection.
%

%

\end{proof}

\begin{thm}\label{thm2.8}
	Let $\widetilde{\Phi}$ be the restriction of $\Phi$ on $\S_n(231)$. Then $\widetilde{\Phi}$ is a bijection from $\S_n(231)$ to $\S_n(321)$.
	Moreover, for $\sigma\in\S_n(231)$, 
	we have 
	\begin{align}\label{Paeq:key}
	&(\cab, \des,\asc, \lda-\fmax, \ldd, \lvalley, \lpeak,\fmax)\,\sigma\nonumber\\
	=&
	( \icr, \drop, \exc+\fix, \cda, \cdd, \cvalley,\cpeak,\fix)\,\widetilde{\Phi}(\sigma)\\
	=&
	(  \cros, \exc, \drop+\fix, \cdd, \cda, \cvalley,\cpeak,\fix)(\widetilde{\Phi}(\sigma))^{-1}.\nonumber
	\end{align}
\end{thm}

\begin{proof}
	For $\sigma\in\S_n(231)$, 
	we have $(\bca)_i=0$ for $i\in [n]$.
	So the inversion bottom (resp. top) number of each letter in $f'$ (resp. $g'$)
	equals 0. Let $\tau=\widetilde{\Phi}(\sigma)$.  
	By definition of $\Phi$ (cf. Section~\ref{def-Phi})  
	the  letters  in $f'$ (resp. $g'$) are in increasing  order. 
	It is not hard to verify that 
		$\nest_i(\tau)=0$ for each $i\in [n]$. 
	By Lemma \ref{321:nest}, we derive that  $\tau\in \S_n(321)$. 
	For $\sigma_1, \sigma_2\in\S_n(231)$,
	since $\Phi$ is a bijection,
	we have $\widetilde{\Phi}(\sigma_1)\neq\widetilde{\Phi}(\sigma_2)$. And $|\S_n(231)|=|\S_n(321)|=C_n$,
	so $\widetilde{\Phi}$ is a bijection from $\S_n(231)$ to $\S_n(321)$.
	Finally,  the equidistribution  \eqref{Paeq:key}  follows from Lemma \ref{97csz}.
\end{proof}

\begin{thm}\label{thm2.9} Let $\widetilde{\Psi}$ be the restriction of $\Psi$ on
	$\S_n(213)$. Then 
	$\widetilde{\Psi}$ is a bijection from $\S_n(213)$ to $\S_n(321)$.
\end{thm}

\begin{proof}
	If  $\sigma\in\S_n(213)$, 
	then $(\bac)_i\sigma=0$ for $i\in [n]$. Thus, $(\bca)_1\hat{\sigma}=0$, and  by Lemma~\ref{lem-thm2.9}, $(\bca)_{i+1}\hat{\sigma}=1$ if $i+1$ is 
	a nondescent top 
	and $(\bca)_{i+1}\hat{\sigma}=0$ otherwise.
	By definition of $\Phi$,  we construct two biwords, $f \choose f'$ and $g \choose g'$,  where 
	$f$ (resp. $g$) is the subword of descent bottoms (resp. nondescent bottoms) in $\hat{\sigma}$ ordered increasingly, and $f'$ (resp. $g'$) is the permutation of descent tops (resp. nondescent tops) in $\hat{\sigma}$ such that the   letters  (resp. except $1$ at the end) in $f'$ (resp. $g'$) are  in increasing order.
	
	 Let  $\tau=\Phi(\hat{\sigma})$.  It is not hard to verify that   $\nest_i(\tau)=1$
	if $i\in g'\setminus\{1\}$ and $\nest_i(\tau)=0$ otherwise. 
	Thus, by \eqref{nestnest},  we have $\nest(\widetilde{\Psi}(\sigma))=0$. By 
	Lemma~\ref{321:nest}, 
	$\widetilde{\Psi}(\sigma)\in \S_n(321)$.
	For $\sigma_1, \sigma_2\in\S_n(213)$,
	since $\Psi$ is a bijection,
	we have $\widetilde{\Psi}(\sigma_1)\neq\widetilde{\Psi}(\sigma_2)$. And $|\S_n(213)|=|\S_n(321)|=C_n$,
	so $\widetilde{\Psi}$ is a bijection from $\S_n(213)$ to $\S_n(321)$.
\end{proof}

\begin{example}
      		If $\sigma=1~6~ 8~ 9~ 7~2~5~3~4$, then $\hat{\sigma}=2~7~9~10~8~3~6~4~5~1$, and reading from left to right,  we obtain the corresponding 
numbers 
		$(\bca)_i:1,1,1,0,0,1,0,1,0,0$  for $i=2, 7,\ldots, 1$, and 
        		$$
        		{f\choose f'}
        		= \left(
        		{1 \atop 5}~
        		{3 \atop 6}~
        		{4 \atop 8}~
        		{8 \atop 10}
        		\right),
        		\quad
        		{g\choose g'}
        		= \left(
        		{2 \atop 2}~
        		{5 \atop 3}~
        		{6 \atop 4}~
        		{7 \atop 7}~
		      {9 \atop 9}~		
        		{10 \atop 1}
        		\right).
        		$$
        		Hence
        		$$
        	w
        		= \left( {f \atop f'}~{g \atop g'} \right)
        		= \left(
        		{1 \atop 5}~
        		{3 \atop 6}~
        		{4 \atop 8}~
        		{8 \atop 10}~
        		{2 \atop 2}~
        		{5 \atop 3}~
        		{6 \atop 4}~
        		{7 \atop 7}~
        		{9 \atop 9}~
        		{10 \atop 1}
        		\right)
        		\to
        		\left(
        		{10 \atop 1}~
        		{2 \atop 2}~
        		{5 \atop 3}~
        		{6 \atop 4}~
        		{1 \atop 5}~
        		{3 \atop 6}~
        		{7 \atop 7}~
        		{4 \atop 8}~
        		{9 \atop 9}~
        		{8 \atop 10}
        		\right).
        		$$
        		Thus $\tau={\Phi}(\hat{\sigma})= 10~2~5~6~1~3~7~4~9~8$, and $\widetilde{\Psi}(\sigma)=\tau(2)\dots\tau(10)=2~5~6~1~3~7~4~9~8$.
        	\end{example}

Combining Theorems~\ref{thm2.7}, \ref{thm2.8}, \ref{thm2.9} and Figure~\ref{FVBFZB} we obtain the diagrams in Figure~\ref{pattern-digram}.

\begin{figure}
\begin{center}

\begin{tikzpicture}[scale=0.3]

\node (LH) at (-12,10) {$\2MP_n$};

\node (SL) at (-20,0) {$\S_{n+1}(213)$};

\node (SC) at (-5,0) {$\S_{n+1}(321)$};

\node (1) at (-18, 5){$\widetilde{\psi}_{{FV}}$};

\node (2) at (-6, 5){$\widetilde{\psi}_{YZL}$};

\node (3) at (-12, 1) {$\widetilde{\Psi}$}; 
\begin{scope}[font=\footnotesize,->]

\draw (SL) -- (LH);

\draw (SC) -- (LH);

\draw (SL) --(SC);

\end{scope}
\end{tikzpicture}
\hspace{1cm}
\begin{tikzpicture}[scale=0.3]
\node (LH) at (-12,10) {$\2MP_n^*$};
\node (SL) at (-20,0) {$\S_{n}(231)$};

\node (SC) at (-5,0) {$\S_{n}(321)$};

\node (1) at (-18, 5){$\widetilde{\phi}_{{FV}}$};

\node (2) at (-6, 5){$\widetilde{\phi}_{{FZ}}$};

\node (3) at (-12, 1) {$\widetilde{\Phi}$}; 
\begin{scope}[font=\footnotesize,->]

\draw (SL) --  (LH);

\draw (SC) -- (LH);

\draw (SL) -- (SC);

\end{scope}
\end{tikzpicture}
\end{center}
\caption{\label{pattern-digram}Two factorizations: $\widetilde{\psi}_{FV}=\widetilde{\psi}_{YZL}\circ \widetilde{\Psi}$  
and $\widetilde{\phi}_{FV}=\widetilde{\phi}_{FZ}\circ \widetilde{\Phi}$}
\label{pattern-diagram}
\end{figure}
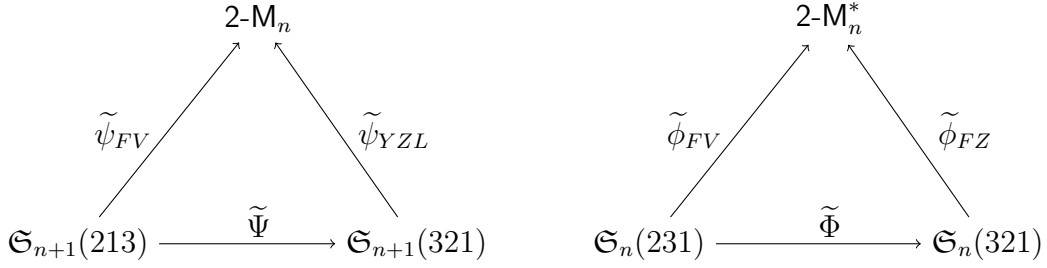

\section{Main results}

For  a finite set of permutations  $\Omega$ and  
$m$ statistics  $\stat_1,  \ldots, \stat_m$  on $\Omega$, we define the generating polynomial 
\begin{align}
P^{(\stat_1,  \ldots,  \stat_m)}(\Omega; t_1, \ldots,  t_m):=&\sum_{\sigma\in\Omega}
t_1^{\stat_1\,\sigma}\,\ldots\, t_m^{\stat_m\,\sigma}.
\end{align}

We define the polynomial
\begin{align}\label{C1eq:dfA1-spec}
A_n(p,q,t) := 
\sum_{\sigma\in \S_n} p^{\nest\, \sigma} q^{\cros \,\sigma}  t^{\exc\,\sigma}.
\end{align}
The following is a 
generalization of Stembridge's identity~\eqref{eq:stembridge}.
 \begin{thm}\label{pqpk}
 For $n\geq1$, we have 
\begin{equation}\label{eqpk1}
A_{n}(p, q, t)=\left(\frac{1+xt}{1+x}\right)^{n-1}P^{(\nest, \cros,\cpk^*,\exc)}
\left(\S_n; p, q, \frac{(1+x)^{2}t}{(x+t)(1+xt)},\frac{x+t}{1+xt}\right),
\end{equation}
equivalently, 
\begin{equation}
P^{(\nest, \cros, \cpk^*,\exc)}(\S_n;  p, q, x,t)=\left(\frac{1+u}{1+uv}\right)^{n-1}A_{n}(p, q, v),\label{eupk3}
\end{equation}
where $u=\frac{1+t^{2}-2xt-(1-t)\sqrt{(1+t)^{2}-4xt}}{2(1-x)t}$ and
$v=\frac{(1+t)^{2}-2xt-(1+t)\sqrt{(1+t)^{2}-4xt}}{2xt}.$
\end{thm}

By Corollary~\ref{starpqzhuang}, we obtain the following linear 
generalization of Stembridge's identity.
 \begin{cor}\label{cor:pqpk}
 For $n\geq1$, we have
\begin{equation}\label{eqp1bis}
A_{n}(p, q, t)=\left(\frac{1+xt}{1+x}\right)^{n-1}P^{(\bac, \cab,\peak-1,\des)}\left(\S_n; p, q, \frac{(1+x)^{2}t}{(x+t)(1+xt)},\frac{x+t}{1+xt}\right),
\end{equation}
equivalently, 
\begin{equation}
P^{(\bac, \cab, \peak-1,\des)}( \S_n; p, q, x,t)=\left(\frac{1+u}{1+uv}\right)^{n-1}A_{n}(p, q, v),\label{linear-eupk3}
\end{equation}
where $u=\frac{1+t^{2}-2xt-(1-t)\sqrt{(1+t)^{2}-4xt}}{2(1-x)t}$ and
$v=\frac{(1+t)^{2}-2xt-(1+t)\sqrt{(1+t)^{2}-4xt}}{2xt}.$
\end{cor}

\begin{remark}  
 When $x=1$ or $p=q=1$
we recover two special cases of \eqref{eqpk1} due to 
 Br\"{a}nd\'{e}n\cite[Eq (5.1)]{Bra08} and  Zhuang~\cite[Theorem $4.2$]{Zhuang17}, respectively.
\end{remark}

With Lemma \ref{general-231-321} and \eqref{peak-val}, letting $p=0$ (resp. $q=0$) in Corollary \ref{cor:pqpk}, we obtain the following corollary.

\begin{cor}\label{pqpk231}
For all positive integers $n$ and each 
triple statistic 
\begin{align*}
(\tau, \stat_1, \stat_2) \in &\{(213, \cab, \valley),
(312, \bac, \valley)\},
\end{align*}
 we have
\begin{equation}\label{eq231pk1}
\begin{split}
&P^{(\stat_1, \des)}(\S_n(\tau); q, t)\\
=&\left(\frac{1+xt}{1+x}\right)^{n-1}P^{(\stat_1, \stat_2,\des)}\left(\S_n(\tau); q,  \frac{(1+x)^{2}t}{(x+t)(1+xt)}, \frac{x+t}{1+xt}\right),
\end{split}
\end{equation}
equivalently, 
\begin{equation}
P^{(\stat_1, \stat_2,\des)}(\S_n(\tau); q, x,t)=\left(\frac{1+u}{1+uv}\right)^{n-1}P^{(\stat_1, \des)}(\S_n(\tau); q, v),\label{eu231pk2}
\end{equation}
where $u=\frac{1+t^{2}-2xt-(1-t)\sqrt{(1+t)^{2}-4xt}}{2(1-x)t}$ and
$v=\frac{(1+t)^{2}-2xt-(1+t)\sqrt{(1+t)^{2}-4xt}}{2xt}$.
\end{cor}

\begin{remark}
When $x=1$, \eqref{eq231pk1} reduces to \cite[Eqs. (1.5) and (1.6)]{FTHZ19}. 
When $(\tau, \stat_1, \stat_2)=(213, \cab, \valley)$ and $q=1$,  \eqref{eq231pk1} reduces to \cite[Corollary $5.3$]{Zhuang17}.
\end{remark}

From  \eqref{INV} and  \eqref{eqpk1} we derive the following result, which is an extension of Shin and Zeng~\cite[Theorem 1]{SZ16}.
\begin{cor}
For $n\geq 1$, 
\begin{equation}\label{symexc1}
\begin{split}
&\sum_{\sigma\in\S_n}q^{\inv\,\sigma-\exc\,\sigma}t^{\exc\,\sigma}\\
=&\left(\frac{1+xt}{1+x}\right)^{n-1}P^{(\bac, \cab,\peak-1,\des)}\left( \S_n;q^2, q, \frac{(1+x)^{2}t}{(x+t)(1+xt)},\frac{x+t}{1+xt}\right).
\end{split}
\end{equation}
\end{cor}

Define the cycle-refinement  of the Eulerian polynomial $A_{n}(t)$ by
$$
A_n^{(\cyc^*-\fix^*, \exc)}(q, t):=\sum_{\sigma\in\S_n}q^{(\cyc^*-\fix^*)\,\sigma}t^{\exc\,\sigma}, 
$$
we obtain 
 a cyclic analogue of Zhuang's formula~\cite[Theorem $4.2$]{Zhuang17}.  
\begin{thm}
\label{t-Aqcpkexc} For $n\geq1$, we have
\begin{equation}\label{Aqexcpk1}
\begin{split}
&A_{n}^{(\cyc^*-\fix^*, \exc)}(q, t)\\
=&\left(\frac{1+xt}{1+x}\right)^{n-1}P^{(\cyc^*-\fix^*, \cpeak^*,\exc)}\left(\S_n; q, \frac{(1+x)^{2}t}{(x+t)(1+xt)},\frac{x+t}{1+xt}\right),
\end{split}
\end{equation}
equivalently, 
\begin{equation}
P^{(\cyc^*-\fix^*, \cpeak^*,\exc)}(\S_n; q, x,t)=\left(\frac{1+u}{1+uv}\right)^{n-1}A_{n}(q, v),\label{depk3}
\end{equation}
where $u=\frac{1+t^{2}-2xt-(1-t)\sqrt{(1+t)^{2}-4xt}}{2(1-x)t}$ and
$v=\frac{(1+t)^{2}-2xt-(1+t)\sqrt{(1+t)^{2}-4xt}}{2xt}.$
\end{thm}

Let $p=q=1$ in \eqref{eqpk1} or $q=1$ in \eqref{Aqexcpk1}, we get the following corollary.

\begin{cor}
\label{t-Acpkexc} For $n\geq1$, we have
\begin{equation}
A_{n}(t)=\left(\frac{1+xt}{1+x}\right)^{n-1}P^{(\cpeak^*,\exc)}\left(\S_n; \frac{(1+x)^{2}t}{(x+t)(1+xt)},\frac{x+t}{1+xt}\right),\label{Aexcpk1}
\end{equation}
equivalently, 
\begin{equation}
P^{(\cpeak^*,\exc)}(\S_n; x,t)=\left(\frac{1+u}{1+uv}\right)^{n-1}A_{n}(v),\label{depk4}
\end{equation}
where $u=\frac{1+t^{2}-2xt-(1-t)\sqrt{(1+t)^{2}-4xt}}{2(1-x)t}$ and
$v=\frac{(1+t)^{2}-2xt-(1+t)\sqrt{(1+t)^{2}-4xt}}{2xt}.$
\end{cor}

 Recall that a permutation $\sigma\in \S_n$ is a derangement if it has no fixed points, i.e., $\sigma(i)\neq i$ for all $i\in [n]$. 
%
%
%
%
Let
$$
D_n^{(\stat_1, \stat_2)}(q, t):=\sum\limits_{\sigma \in \mathfrak{D}_n} q^{\stat_1\,\sigma} t^{\stat_2\,\sigma},
$$ 
where $\D_n$ is  the set of $derangements$ in $\S_n$. 


Taking $(p, q, tq, r)=(q, 1,t, 0)$ (resp. $(p, q, tq, r)=(q^2, q, tq, 0)$) in
Theorem \ref{thm:cpeakS_1} and by \eqref{INV}, we obtain the following corollary. 
\begin{cor}\label{3.6}
For all positive integers $n$ and for each statistic $\stat\, \in \{\nest, \inv 
 \}$,  
\begin{equation}\label{exc3qpk2}
D_{n}^{(\stat, \exc)}(q, t)=\left(\frac{1+xt}{1+x}\right)^{n}P^{(\stat, \cpeak,\exc)}\left(\D_n; q, \frac{(1+x)^{2}t}{(x+t)(1+xt)},\frac{x+t}{1+xt}\right),
\end{equation}
equivalently, 
\begin{equation}
P^{(\stat, \cpeak,\exc)}(\D_n; q, x,t)=\left(\frac{1+u}{1+uv}\right)^{n}D_{n}^{(\stat, \exc)}(q, v),\label{depk2}
\end{equation}
where $u=\frac{1+t^{2}-2xt-(1-t)\sqrt{(1+t)^{2}-4xt}}{2(1-x)t}$ and
$v=\frac{(1+t)^{2}-2xt-(1+t)\sqrt{(1+t)^{2}-4xt}}{2xt}.$
\end{cor}

 By \eqref{INV} and Lemma \ref{321:nest}, the $r=0$ case of  \eqref{321cpeakS_1}
yields the following result in parallel with Corollary~\ref{3.6}, which generalizes 
Lin's identity~ \cite[Theorem $1.4$]{Lin17}.
\begin{cor}
For $n\geq 1$,
\begin{align}\label{eq:SZdn321}
&P^{(\inv, \exc)}(\D_n(321); q, t)\\
&=\left(\frac{1+xt}{1+x}\right)^{n}P^{(\inv, \cpeak,\exc)}\left(\D_n(321); q, \frac{(1+x)^{2}t}{(x+t)(1+xt)},\frac{x+t}{1+xt}\right),\nonumber
\end{align}
equivalently, 
\begin{equation}
P^{(\inv, \cpeak,\exc)}(\D_n(321); q, x,t)=\left(\frac{1+u}{1+uv}\right)^{n}P^{(\inv, \exc)}(\D_n(321); q, v),\label{depk2-321}
\end{equation}
where $u=\frac{1+t^{2}-2xt-(1-t)\sqrt{(1+t)^{2}-4xt}}{2(1-x)t}$ and
$v=\frac{(1+t)^{2}-2xt-(1+t)\sqrt{(1+t)^{2}-4xt}}{2xt}.$
\end{cor}

Moreover, we have the  the following formula. 
\begin{thm} \label{thm:SZ3q}
For all positive integers $n$, 
%
\begin{equation}\label{exc3qpk1}
D_{n}^{(\cyc, \exc)}(q, t)=\left(\frac{1+xt}{1+x}\right)^{n}P^{(\cyc, \cpeak,\exc)}\left(\D_n; q, \frac{(1+x)^{2}t}{(x+t)(1+xt)},\frac{x+t}{1+xt}\right),
\end{equation}
equivalently, 
\begin{equation}
P^{(\cyc, \cpeak,\exc)}(\D_n; q, x,t)=\left(\frac{1+u}{1+uv}\right)^{n}D_{n}^{(\cyc, \exc)}(q, v),\label{depk5}
\end{equation}
where $u=\frac{1+t^{2}-2xt-(1-t)\sqrt{(1+t)^{2}-4xt}}{2(1-x)t}$ and
$v=\frac{(1+t)^{2}-2xt-(1+t)\sqrt{(1+t)^{2}-4xt}}{2xt}.$
\end{thm}

\begin{thm} \label{thm:cpeakS_1}
For $n\geq 1$,  
\begin{align}\label{cpeakS_1}
&P^{(\nest, \cros,  \exc, \fix)}(\S_n; p, q, tq, r)\nonumber\\
=&\left(\frac{1+xt}{1+x}\right)^{n}P^{(\nest, \cros,\cpeak, \exc, \fix)}
\left(\S_n; p, q, \frac{(1+x)^{2}t}{(x+t)(1+xt)},
\frac{q(x+t)}{1+xt}, \frac{(1+x)r}{1+xt}\right),
\end{align}
equivalently, 
\begin{align}
&P^{(\nest, \cros, \cpeak, \exc, \fix)}(\S_n; p, q, x, qt, r)\nonumber\\
&=\left(\frac{1+u}{1+uv}\right)^{n}P^{(\nest, \cros,  \exc, \fix)}\left(\S_n; p, q, qv, \frac{(1+uv)r}{1+u}\right),\label{cpeakS_2}
\end{align}
where $u=\frac{1+t^{2}-2xt-(1-t)\sqrt{(1+t)^{2}-4xt}}{2(1-x)t}$ and
$v=\frac{(1+t)^{2}-2xt-(1+t)\sqrt{(1+t)^{2}-4xt}}{2xt}.$
\end{thm}

\begin{remark}
   Cooper et al.~\cite[Theorem~11]{CJZ19} have recently proved
  the $p=q=1$ case of \eqref{cpeakS_1} by applying  Sun and 
Wang's CMFS action~\cite{SW14}, see \eqref{def:cmfs}.
\end{remark}

Applying   Lemma \ref{321:nest} and  Theorem \ref{thm:cpeakS_1} with $p=0$,
we obtain the following result.
\begin{cor} \label{thm:321cpeakS_1}
	For $n\geq 1$,  
	\begin{align}\label{321cpeakS_1}
	&P^{(\cros, \exc, \fix)}(\S_n(321); q, tq, r)\nonumber\\
	=&\left(\frac{1+xt}{1+x}\right)^{n}P^{(\cros, \cpeak, \exc, \fix)}\left(\S_n(321); q, \frac{(1+x)^{2}t}{(x+t)(1+xt)},\frac{q(x+t)}{1+xt}, \frac{(1+x)r}{1+xt}\right),
	\end{align}
	equivalently, 
	\begin{align}
	&P^{(\cros, \cpeak, \exc, \fix)}(\S_n(321); q, x, qt, r)\nonumber\\
	=&\left(\frac{1+u}{1+uv}\right)^{n}P^{(\cros, \exc, \fix)}\left(\S_n(321); q, qv, \frac{(1+uv)r}{1+u}\right),
	\end{align}
	where $u=\frac{1+t^{2}-2xt-(1-t)\sqrt{(1+t)^{2}-4xt}}{2(1-x)t}$ and
	$v=\frac{(1+t)^{2}-2xt-(1+t)\sqrt{(1+t)^{2}-4xt}}{2xt}.$
\end{cor}

Consider  the generalized $q$-Narayana polynomials $N_n(t,q, r)$ defined by
\begin{align}\label{def:N}
N_n(t, q, r):=\sum\limits_{\sigma \in \S_n(321)}t^{\exc\,\sigma} q^{\inv\,\sigma} r^{\fix\,\sigma}.
\end{align}
In particular, we have
\begin{align}
N_n(t/q, q, 1)&=\sum_{\sigma\in\S_n(321)}t^{\exc\,\sigma}
q^{\inv\,\sigma-\exc\,\sigma},\label{N1}\\
N_n(t, q, t)&=\sum_{\sigma\in\S_n(321)}t^{\wex\,\sigma}q^{\inv\,\sigma}.\label{N2}
\end{align}
Fu et al. \cite{FTHZ19} gave more interpretations of $N_{n}(t/q, q, 1)$ and $N_{n}(t, q, t)$ in terms of $n$-permutation  patterns.
We further  prove the following interpretations by using the $(n-1)$-permutation patterns.
\begin{thm}\label{qnara}
	For $n\geq 1$, the following identities hold
	\begin{align}
	N_{n}(t/q, q, 1)&=\sum_{\sigma\in\S_{n-1}(\tau)}t^{\stat_1\:\sigma}q^{\stat_2\:\sigma}(1+t)^{\stat_3\:\sigma},\label{N1new}\\
	N_{n}(t, q, t)&=t^n\sum_{\sigma\in\S_{n-1}(\tau)}(q/t)^{\stat_1\:\sigma}q^{\stat_2\:\sigma}(1+q/t)^{\stat_3\:\sigma},\label{N2new}
    \end{align}
    where five choices for the quadruples $(\tau,\stat_1,\stat_2,\stat_3)$ are listed in Table~\ref{five}.
\end{thm}

\begin{table}[tbp]\caption{Five choices of $(\tau,\stat_1,\stat_2,\stat_3)$}\label{five}
		\centering
		\begin{tabular}{|c|c|c|c|c|c|}
			\hline
			\# & $\tau$ & $\stat_1$  & $\stat_2$ & $\stat_3$\\
			\hline
			1 & $321$ & $\exc$ & $\inv$ & $\fix$\\
			2 & $231$ & $\des$ & $\des+\cab$ & $\fmax$\\
			3 & $132$ & $\asc$ & $\asc+\bac$ & $\amax$\\
			4 & $312$ & $\des$ & $\des+\bca$ & $\amin$\\
			5 & $213$ & $\asc$ & $\asc+\acb$ & $\fmin$\\
			\hline
		\end{tabular}
	\end{table}

%
%

For $0\leq k\leq n$, define the sets
\begin{subequations}
\begin{align}\label{sndedef} 
\widetilde{\S}_{n,k}(321) &=\{\sigma\in\S_{n}(321):\exc\,\sigma=k,~\cda\,\sigma=0\},\\
\widetilde{\S}_{n,k}(213)&=\{\sigma\in \S_{n}(213): \asc\:\sigma=k, ~\rda\:\sigma=0  \},\\
\widetilde{\S}_{n,k}(312)&=\{\sigma\in \S_{n}(312): \des\, \sigma=k, ~\ldd\:\sigma=0  \},\\
\widetilde{\S}_{n,k}(132)&=\{\sigma\in \S_{n}(132): \asc\:\sigma=k, ~\rda\:\sigma=0 \},\\
\widetilde{\S}_{n,k}(231)&=\{\sigma\in \S_{n}(231): \des\:\sigma=k, ~\ldd\:\sigma=0 \},
\end{align}
\end{subequations}
and $\widetilde{\S}_{n}(\tau) =\cup_{k=0}^{n} \widetilde{\S}_{n,k}(\tau)$ for $\tau\in \S_3$.

 
\begin{thm}\label{thm:general-qgamma}
For  $n\geq1$,  the following  
$q$-analogue of \eqref{nara-gamma} holds
\begin{subequations}
\begin{align}\label{general-qgamma}
N_n(t/q, q, 1)=
\sum_{k=0}^{\lfloor \frac{n-1}{2}\rfloor}
\gamma_{n-1, k}(q)t^k(1+t)^{n-1-2k},
\end{align}
where 
\begin{align}
\gamma_{n-1, k}(q)&=\sum_{\pi\in\widetilde{\S}_{n-1,k}(321)}q^{\inv\,\pi}\label{eq:inv321}\\
&=\sum_{\pi\in\widetilde{\S}_{n-1,k}(231)}q^{(\cab)\:\pi+\des\,\pi}
=\sum_{\pi\in\widetilde{\S}_{n-1,k}(312)}q^{(\bca)\:\pi+\des\,\pi}\label{eq:312:231}\\
&=\sum_{\pi\in\widetilde{\S}_{n-1,k}(132)}q^{(\bac)\:\pi+\asc\,\pi}
=\sum_{\pi\in\widetilde{\S}_{n-1,k}(213)}q^{(\acb)\:\pi+\asc\,\pi}
\label{eq:213:132}.
\end{align}
\end{subequations}
\end{thm}

\begin{thm}\label{qwexth}
For $n\geq 1$,  the following  
$q$-analogue of \eqref{nara-gamma} holds
\begin{equation}\label{narawex}
N_n(t, q, t)=
\sum_{k=1}^{\lfloor\frac{n+1}{2}\rfloor}\widetilde\gamma_{n-1, k-1}(q)t^k(1+t/q)^{n+1-2k},
\end{equation}
where
\begin{align}
\widetilde{\gamma}_{n-1, k-1}(q)
&=\sum_{\pi\in\widetilde{\S}_{n-1,k}(321)}q^{n-1+\inv\,\pi-\exc\,\pi}\\
&=\sum_{\pi\in\widetilde{\S}_{n-1,k-1}(231)}q^{n-1+(\cab)\:\pi}
=\sum_{\pi\in\widetilde{\S}_{n-1,k-1}(312)}q^{n-1+(\bca)\:\pi}\label{eq:wex312:231}\\
&=\sum_{\pi\in\widetilde{\S}_{n-1,k-1}(132)}q^{n-1+(\bac)\:\pi}
=\sum_{\pi\in\widetilde{\S}_{n-1,k-1}(213)}q^{n-1+(\acb)\:\pi}\label{eq:wex213:132}.
\end{align}
\end{thm}
\begin{remark}
Other  interpretations for $\gamma_{n-1, k}(q)$ and 
$\widetilde{\gamma}_{n-1, k-1}(q)$ are  given in 
\cite{LF17, Lin17, FTHZ19}.
\end{remark}

Let $\B_n$ be the set of permutations $\sigma$  of $\{\pm 1,\ldots, \pm n\}$ with $\sigma(-i)=-\sigma(i)$ for every $i\in [n]$.
From  Steingr\'{i}msson~\cite[Definition 3]{Ste94},  we define the \emph{excedance} of $\sigma \in \B_n$ by $i <_{f} \sigma(i)$ for $i\in[n]$,
in the \emph{friends order} $<_f$ of $\{\pm 1,\ldots, \pm n\}$:
$$
1<_{f}-1<_{f}2<_{f}-2<_{f} \cdots <_{f} n <_{f} -n,
$$
and denote the number of excedances of $\sigma $ by $\exc_B(\sigma)$.
Following  Brenti~\cite{Bre94} we  say that $i\in \{0,1,\ldots, n-1\}$ is a \emph{B-descent} of $\sigma$ if
$\sigma(i) > \sigma(i+1)$
in the natural order $<$ of $\{\pm 1,\ldots, \pm n\}$: 
$$
-n<\cdots <-2<-1<1<2<\cdots <n, 
$$
where $\sigma(0)=0$. Denote the number of $B$-descents of $\sigma$ by $\des_B(\sigma)$.  Brenti~\cite[Theorem~3.4]{Bre94}  considered the   Eulerian polynomials  of type B
\begin{align}
B_n(y,t):=\sum_{\sigma\in \B_n}y^{\neg\,\sigma}t^{\des_B\,\sigma}
\end{align}
and proved the following exponential g.f. 
\begin{align}\label{brenti}
\sum_{n\ge 0} {B}_n(y,t) \frac{z^n}{n!}
&=\frac{(1-t) e^{z(1-t)}}{1-te^{z(1-t)(1+y)}}\nonumber\\
&=e^{y (t-1)z}S((1+y)z;t),
\end{align}
where $S(z;t):=\frac{(1-t)e^{z(1-t)}}{1-te^{z(1-t)}}$ is the exponential g.f. of type A Eulerian polynomials $A_n(t)$. 

Our main results for the  polynomials $B_n(y,t)$  are the following two theorems.

\begin{thm}\label{thmBgfexc}
We have 
\begin{equation}\label{Bequidistri}
B_n(y,t)=\sum_{\sigma\in \B_n}y^{\neg\,\sigma}t^{\exc_B\,\sigma}.
\end{equation} 
\end{thm}
\begin{thm}\label{Bnegexc}
For $n\geq 1$, 
\begin{subequations}
\begin{equation}\label{eq:extended-zhuang}
B_n(y, t)=(1+yt)^nP^{(\cpeak, \exc)}\left(\S_n; \frac{(1+y)^2t}{(y+t)(1+yt)}, \frac{y+t}{1+yt}\right),
\end{equation}
equivalently, 
\begin{equation}
P^{(\cpeak,\exc)}\left(\S_n; y,t\right)=\frac{1}{(1+uv)^{n}}B_{n}(u,v),\label{e-lpkdesb2}
\end{equation}
\end{subequations}
where $u=\frac{1+t^{2}-2yt-(1-t)\sqrt{(1+t)^{2}-4yt}}{2(1-y)t}$ and
$v=\frac{(1+t)^{2}-2yt-(1+t)\sqrt{(1+t)^{2}-4yt}}{2yt}.$
\end{thm}

\section{Proofs using group actions}
In this section, using group actions  we shall prove Theorem~\ref{thm:SZ3q} and Theorem~\ref{thm:cpeakS_1},  respectively, in the following two  subsections. 

\subsection{Proof of Theorem~\ref{thm:SZ3q}}

Let $\sigma\in\S_n$ with convention 0--$\infty$.
 For  any  $x\in[n]$, the {\em$x$-factorization} of $\sigma$ reads $\sigma=w_1 w_2x w_3 w_4,$ where $w_2$ (resp.~$w_3$) is the maximal contiguous subword immediately to the left (resp.~right) of $x$ whose letters are all smaller than $x$.  Following Foata and Strehl~\cite{FS74} we define the action $\varphi_x$ by
$$
\varphi_x(\sigma)=w_1 w_3x w_2 w_4.
$$
Note that if $x$ is a double ascent (resp. double descent), then $w_3=\varnothing$ (resp. $w_2=\varnothing$), and if $x$ is a valley then
$w_2=w_3=\varnothing$.
For instance, if $x=5$ and $\sigma=26471583\in\S_7$, then $w_1=2647,w_2=1,w_3=\varnothing$ and $w_4=83$.
Thus $\varphi_5(\sigma)=26475183$.
Clearly, $\varphi_x$ is an involution acting on $\S_n$ and it is not hard to see that $\varphi_x$ and $\varphi_y$ commute for all $x,y\in[n]$. Br\"and\'en~\cite{Bra08} modified the map $\varphi_x$ to be
\begin{align*}
\varphi'_x(\sigma):=
\begin{cases}
\varphi_x(\sigma)&\text{if $x$ is not a  peak of $\sigma$},\\
\sigma& \text{if $x$ is a peak  of $\sigma$.}
\end{cases}
\end{align*}
It is clear that $\varphi'_x$ is involution and commutes with $\varphi'_y$ for  $x\neq y$. For any subset $S\subseteq[n]$ with $S=\{x_1, \ldots, x_r\}$ we then define the map $\varphi'_S :\S_n\rightarrow\S_n$ by
\begin{align*}
\varphi'_S(\sigma)=\prod_{x\in S}\varphi'_x(\sigma)
\end{align*}
where $\prod_{x\in S}\varphi_x'=\varphi'_{x_1}\circ \cdots \circ \varphi'_{x_r}$.
Hence the group $\Z_2^n$ acts on $\S_n$ via the functions $\varphi'_S$, $S\subseteq[n]$. This action is called  the {\em Modified Foata--Strehl action} ({\em MFS-action} for short). 

Recall that  a permutation $\sigma\in\S_n$ can be    factorized into 
 distinct cycles, say $C_1, C_2, \cdots, C_k$, where 
each  cycle $C$ can be written  as a sequence $C=(a, \sigma(a), \ldots, \sigma^{r-1}((a))$ 
with  $\sigma^{r}(a)=a$ for some $a, r\in [n]$.
We  say that $\stan(\sigma):=C_1C_2\cdots C_k$  is  the  \emph{standard cycle representation} of $\sigma$ if
\begin{itemize}
\item the largest element of each cycle is at the first position,
\item  the cycles are arranged   in increasing order according to their largest elements.
\end{itemize}
We define $\iota(\sigma)$ to be the permutation obtained from $\stan(\sigma)$ by 
erasing the parentheses of cycles. 
For example, for  $\sigma=26471583\in \D_8$, then $\stan(\sigma)=(6512)(8347)$ and 
 $\iota(\sigma)=65128347$. 

In this section, we consider the statistics of $\iota(\sigma)$ with the convention 
0--$\infty$. 

\begin{lem}\label{Cycstat}
	For $\sigma\in \D_n$,  we have  
	\begin{align*}
	\cvalley\,\sigma=&\lvalley\,\iota(\sigma)=\lpeak\,T(\sigma)=\cpeak\,\sigma,\quad
	\lda\,\iota(\sigma)=\exc\,\sigma-\cpeak\,\sigma,\\
	\ldd\,\iota(\sigma)=&n-\cpeak\,\sigma-\exc\,\sigma,\quad
	\lda\,\iota(\sigma)+\ldd\,\iota(\sigma)=n-2\cpeak\,\sigma.
	\end{align*}
\end{lem}
\begin{proof}
	The first two identities are easily seen by the definitions of $\sigma$ and $\iota(\sigma)$. For the third identity, 
	\begin{align*}
	\ldd\,\iota(\sigma)=&n-(\lpeak\,\iota(\sigma)+\lvalley\,\iota(\sigma)+\lda\,\iota(\sigma))\\
	=&n-(\cpeak\,\sigma+\cvalley\,\sigma+\exc\,\sigma-\cvalley\,\sigma)\\
	=&n-\cpeak\,\sigma-\exc\,\sigma.
	\end{align*}
	With the second and third identities, the fourth identity can be derived directly. 
\end{proof}

For $\sigma\in \D_n$, define the map $\tau^c_x: \D_n\mapsto\D_n$ by
$$\tau^c_x(\sigma):=\iota^{-1}(\varphi_{x}^{\prime}(\iota(\sigma))).$$
It is easy to see that $\tau^c_x$ is an involution and commutes with $\tau^c_y$ for $x, y\in [n]$.
Let $S\subseteq[n]$, we define $\tau_{S}^{c}:\mathfrak{D}_{n}\rightarrow\mathfrak{D}_{n}$
by 
\begin{align}\label{def:cmfs}
\tau_{S}^{c}(\sigma)=\prod_{x\in S}\tau_{x}^{c}(\sigma).
\end{align}
Sun and Wang~\cite{SW14} defined 
the group action of $\mathbb{Z}_{2}^{n}$  on $\mathfrak{D}_{n}$ via the
involutions $\tau_{S}^{c}$ over all $S\subseteq[n]$; this
group action is called the \textit{cyclic modified Foata}--\textit{Strehl
action}, abbreviated \textit{CMFS-action}, see Figure \ref{CMFS} for an illustration.
For any  permutation $\sigma\in\S_{n}$, let $\Orb(\sigma)=\{g(\sigma): g\in\Z_2^{n}\}$ be the \emph{orbit of $\sigma$ under the CMFS-action}.
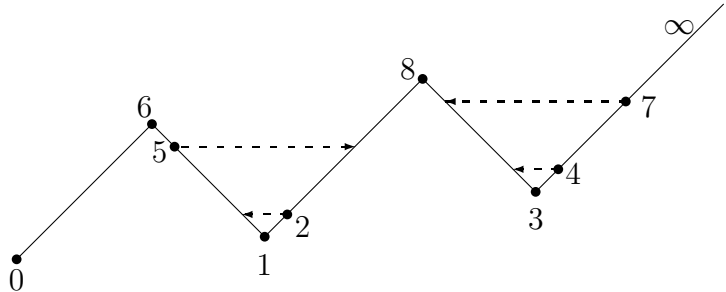
\begin{figure}[htb]
\begin{center}
\setlength {\unitlength} {0.8mm}
\begin {picture} (30,50) \setlength {\unitlength} {1mm}
\thinlines



\put(-15,17){\dashline{1}(1,0)(23,0)}
\put(9,17){\vector(1,0){0.1}}

\put(-1,8){\dashline{1}(1,0)(-4.5,0)}
\put(-6,8){\vector(-1,0){0.1}}

\put(35,14){\dashline{1}(1,0)(-4.5,0)}
\put(30,14){\vector(-1,0){0.1}}

\put(44,23){\dashline{1}(1,0)(-23,0)}
\put(21,23){\vector(-1,0){0.1}}

\put(-37, -2){$0$}
\put(-36, 2){\circle*{1.3}}
\put(-36,2){\line(1,1){18}}
\put(-18, 20){\circle*{1.3}}
\put(-20,21){$6$}
\put(-18,20){\line(1,-1){15}}
\put(-18,15){$5$}
\put(-15, 17){\circle*{1.3}}
\put(-4.1,0){$1$}
\put(-3,5){\circle*{1.3}}
\put(1,5){$2$}
\put(0,8){\circle*{1.3}}
\put(-3,5){\line(1,1){21}}
\put(15,26){$8$}
\put(18, 26){\circle*{1.3}}
\put(32,6){$3$}
\put(33, 11){\circle*{1.3}}
\put(18,26){\line(1,-1){15}}
\put(37,12){$4$}
\put(36, 14){\circle*{1.3}}
\put(47,21){$7$}
\put(45, 23){\circle*{1.3}}\put(33,11){\line(1,1){25}}\put(50,32){$\infty$}


\end{picture}
\end{center}
\caption{CMFS-actions on $(6512)(8347)$
\label{CMFS}}
\end {figure}


\begin{remark}\label{Cha1coeffrela}
The CMFS-action divides the set $\D_{n}$ into disjoint orbits. Moreover, for $\sigma\in \D_{n}$, $x$ is a double drop (resp. double excedance) of $\sigma$ if and only if $x$ is a double excedance (resp.double drop) of $\tau^c_x(\sigma)$. A double drop (resp. double excedance) $x$ of $\sigma$ remains a double drop (resp. double excedance) of $\tau^c_y(\sigma)$ for any $y\neq x$. Hence, there is a unique permutation in each orbit which has no double excedance. 
Let $\check\sigma$ be this unique element in $\Orb(\sigma)$, 
and for any other $\sigma'\in\Orb(\sigma)$, it can be obtained from $\check\sigma$ by repeatedly applying $\tau^c_x$ for some double drop $x$ of $\check\sigma$. Each time this happens, $\exc$ increases by $1$ and $\cdd$ decreases by $1$. Thus by Lemma \ref{Cycstat}, we have  
\begin{align}\label{orbit:excgf}
\sum_{\sigma\in\Orb\:\sigma}t^{\exc\:\sigma}=
t^{\exc\:\check\sigma}(1+t)^{\cdd\:\check\sigma}
=t^{\cpeak\:\check\sigma}(1+t)^{n-2\cpeak\:\check\sigma}.
\end{align}
We obtain 
gamma expansion of derangement polynomials 
immediately by summing over all the orbits that form $\D_n$.
\end{remark}

We can give a more general version of Theorem~\ref{thm:SZ3q}.	
For any subset $\Pi\subseteq\S_n$  let 
\[
A^{(\exc, \cyc)}(\Pi;w,t):=\sum_{\sigma\in\Pi}w^{\cyc\,\sigma}t^{\exc\,\sigma}.
\]
The set $\Pi$
is \emph{invariant} under the CMFS-action 
if $\tau_{S}^{c}(\sigma)\in\Pi$ for any $\sigma\in\Pi$  and any  $S\subseteq[n]$.

\begin{thm}\label{5.15}
	If $\Pi\subseteq\mathfrak{D}_{n}$ is invariant under the CMFS-action,  then 
	\begin{equation}\label{SZ12cyc} 
	A^{(\cyc,\exc)}(\Pi;w,t)=\left(\frac{1+xt}{1+x}\right)^{n}P^{(\cyc,\cpeak,\exc)}\left(\Pi;w,\frac{(1+x)^{2}t}{(x+t)(1+xt)},\frac{x+t}{1+xt}\right),
	\end{equation}
	equivalently, 
	\[
	P^{(\cyc,\cpeak,\exc)}(\Pi;x,t,w)=\left(\frac{1+u}{1+uv}\right)^{n+1}A^{(\cyc,
		\exc)}(\Pi;w,v),
	\]
	where $u=\frac{1+t^{2}-2xt-(1-t)\sqrt{(1+t)^{2}-4xt}}{2(1-x)t}$ and
	$v=\frac{(1+t)^{2}-2xt-(1+t)\sqrt{(1+t)^{2}-4xt}}{2xt}.$ 
\end{thm}
%
%

First we prove the following identity. 
\begin{lem} Let $\sigma\in \mathfrak{D}_{n}$. We have
	\begin{equation}\label{D-orbitc1}
	(1+x)^{\cda\;\sigma+\cdd\;\sigma}
	\sum_{\sigma'\in\Orb(\sigma)}t^{\exc\,\sigma'}
	=\sum_{\sigma'\in\Orb(\sigma)}(1+xt)^{\cdd\,\sigma'}(x+t)^{\cda\,\sigma'}t^{\cvalley\,\sigma'}.
	\end{equation}
\end{lem}
\begin{proof} 
Let $j=\cda\,\sigma+\cdd\,\sigma$. By \eqref{orbit:excgf} the left-hand side of \eqref{D-orbitc1} is equal to
$$
(1+x)^j t^{\cval\, \check\sigma}  (1+t)^j=t^{\cval \,\check\sigma} (1+xt+x+t)^j.
$$
Let $J(\sigma)$ be the set of indices of double excedances and double drops of $\sigma$, i.e.,
	$$
	J(\sigma):=
	\{i\in [n]: \sigma(i)\, \text{is a double excedance or double drop}\}.
	$$
	Clearly $|J(\sigma)|=j$.
By \eqref{def:cmfs}  CMFS-action  establishes a bijection from  the set of subsets of 
$J(\sigma)$ to $\Orb(\sigma)$ such that  if $S\subset J(\sigma)$ then
$|S|=\cdd\,\sigma'$ with  $\sigma'=\tau_{S}^{c}(\sigma)$. Hence the right-hand side of 
 \eqref{D-orbitc1} is equal to
 $$
t^{\cval \check\sigma}\sum_{S\subset J(\sigma)}(1+xt)^{|S|}
	(x+t)^{j-|S|}.
	$$
Eq.~\eqref{D-orbitc1}  follows then from
	$\sum_{S\subset [j]}(1+xt)^{|S|}
	(x+t)^{j-|S|}=
	(1+xt+x+t)^j$.
\end{proof}

\begin{proof}[Proof of  Theorem~\ref{5.15}]
With Lemma \ref{Cycstat} and Eq. \eqref{D-orbitc1}, we have 
\[
\Big(\sum_{\sigma'\in\Orb(\sigma)}t^{\exc\,\sigma'}\Big)(1+x)^{n-2\cpeak\,\sigma'}=\sum_{\sigma'\in\Orb(\sigma)}(1+xt)^{n-\exc\,\sigma'-\cpeak\,\sigma'}(x+t)^{\exc\,\sigma'-\cpeak\,\sigma'}t^{\cpeak\,\sigma'},
\]
which is equivalent to 
\[
\sum_{\sigma'\in\Orb(\sigma)}t^{\exc\,\sigma'}=\sum_{\sigma'\in\Orb(\sigma)}\frac{(1+xt)^{n-\exc\,\sigma'-\cpeak\,\sigma'}(x+t)^{\exc\,\sigma'-\cpeak\,\sigma'}t^{\cpeak\,\sigma'}}{(1+x)^{n-2\cpeak\,\sigma'}}.
\]
Then, summing over all the orbits leads to 
\[
\sum_{\sigma\in\Pi}t^{\exc\,\sigma}=\sum_{\sigma\in\Pi}\frac{(1+xt)^{n-\exc\,\sigma-\cpeak\,\sigma}(x+t)^{\exc\,\sigma-\cpeak\,\sigma}t^{\cpeak\,\sigma}}{(1+x)^{n-2\cpeak\,\sigma}}.
\]
For $\sigma'\in \Orb(\sigma)$, first we have $\cyc(\sigma')=\cyc(\sigma)$.
From the definition of $o(\sigma)$,
we have $\cyc(\sigma)$ is equal to the number of left-to-right maximum of $o(\sigma)$.
It is easy to see that the number of left-to-right maximum is invariant under MFS-action.
Thus the number  $\cyc(\sigma')$ is invariant for any $\sigma'\in \Orb(\sigma)$.
Therefore,
\begin{align*}
\sum_{\sigma\in\Pi}t^{\exc\,\sigma}w^{\cyc\,\sigma}&=\sum_{\sigma\in\Pi}\frac{(1+xt)^{n-\exc\,\sigma-\cpeak\,\sigma}(x+t)^{\exc\,\sigma-\cpeak\,\sigma}t^{\cpeak\,\sigma}}{(1+x)^{n-2\cpeak\,\sigma}}w^{\cyc\,\sigma}
\\
&=\left(\frac{1+xt}{1+x}\right)^{n}P^{(\cpeak,\exc,\cyc)}\left(\Pi;\frac{(1+x)^{2}t}{(x+t)(1+xt)},\frac{x+t}{1+xt},w\right).
\end{align*}
The proof is completed.
\end{proof}

\begin{remark}
Recently, using the joint distribution of the cyclic valley number and excedance number statistics Cooper, Jones and Zhuang~\cite{CJZ19} have
 generalized the formula of Stembridge by applying  Sun and 
Wang's CMFS-action. In particular they also obtained the $w=1$
case of Theorem~\ref{5.15}.
\end{remark}

\subsection{Proof of Theorem~\ref{thm:cpeakS_1}}

For our purpose we define  a 3-\emph{Motzkin path}  of length $n$  as a word $\mathbf{s}:=s_1\ldots s_n$ on the alphabet $\{\su, \sd, \ly, \lb, \lr\}$ such that  $|s_1\ldots s_n|_\su= [s_1\ldots s_n|_\sd$ and the \emph{height} of the $i$th step is nonnegative, i.e., 
		\begin{equation}\label{heigt-condition3}
		h_{i}(\mathbf{s}):=|s_1\ldots s_i|_\su-[s_1\ldots s_i|_\sd\geq 0 \quad (i=1, \ldots, n),
		\end{equation}
	 where $|s_1\ldots s_i|_\su$ is the number of letters $\su$ in the word $s_1\ldots s_i$. 
	 Let
$$
\alpha(\mathbf{s}):=\{i\in [n]: s_i=\alpha\}\quad \textrm{for}\quad \alpha\in \{\su, \sd, \ly, \lb, \lr\}.
$$	 
	 

A \emph{variant restricted Laguerre history} of length  $n$ is  a pair $(\mathbf{s, p})$, where $\mathbf{s}$ is a 3-Motzkin path $s_1\ldots s_n$ and $\mathbf{p}=(p_1, \ldots, p_n)$ with $0\leq p_i\leq h_{i-1}(\mathbf{s})$ if $s_i=\su$, $0\leq p_i\leq h_{i-1}(\mathbf{s})-1$ if $s_i=\sd,\lb,\lr$ and $p_i=h_{i-1}$ if $s_i=\ly$ with $h_0(\mathbf{s})=0$. 
Let $\LH'_n$ be the set of variant restricted Laguerre histories of length $n$.

We use a variant of Foata-Zeilberger's bijection $\phi_{FZ} : \S_{n} \to \LH'_{n}$ (cf. \eqref{resFZ}). Given $\sigma\in \S_{n}$, we construct the variant restricted Laguerre history
$\phi'_{{FZ}}(\sigma):=(\mathbf{s, p})\in \L'_n$ as follows. 
For $i=1, \ldots, n$, let
\begin{align}\label{varresFZ}
s_i=\begin{cases}
\su &\textrm{if} \; i\in \Cval\sigma,
\\
\sd &\textrm{if} \; i\in \Cpeak\sigma,
\\
\la &\textrm{if} \; i\in \Cdd\sigma,
\\
\lb &\textrm{if} \; i\in \Cda\sigma,
\\
\ly &\textrm{if} \; i\in \Fix\sigma,
\end{cases}
\end{align}
with $p_i=\nest_i\sigma$.

\begin{lem}\label{lem: phiFZ} If $\phi'_{FZ}(\sigma)=(\mathbf{s, p})\in \L'_n$ with $\sigma\in \S_n$, then 
	\begin{subequations}
	\begin{align}
	\Fix\,\sigma\,&=\,\ly(\mathbf{s}),\label{grouplag1}\\
	\Exc\,\sigma\,&=\,\lb(\mathbf{s})\,\cup \su(\mathbf{s}),\label{grouplag2}\\
	\nest\,\sigma\,&=\,\sum_{i=1}^np_i,\label{grouplag3}\\
	\exc\,\sigma+\cros\,\sigma+\nest\,\sigma\,&=\,\sum_{i=1}^nh_{i-1}(\mathbf{s}),\label{grouplag4}
	\end{align}
	\end{subequations}
	where $\Exc\,\sigma$ denotes the set of excedances of $\sigma$.
	\end{lem}
\begin{proof}
	From the construction of $\phi'_{FZ}$, it is easy to see \eqref{grouplag1}-\eqref{grouplag3}. 
	Define 
	\[
	\exc_i\sigma=
	\left\{
	\begin{array}{ll}
	1 & \text{if $\sigma(i)>i$}, \\
	0 & \text{if $\sigma(i)\leq i$.} 
	\end{array}
	\right.
	\]
	By inductions on $i\in [n]$ we verify that
	\begin{equation}\label{laglemeq}
	\exc_i\sigma+\nest_{i}\sigma+\cros_{i}\sigma=
	\left\{
	\begin{array}{ll}
	h_{i-1}(\mathbf{s})+1, & \text{if $s_i=\su$}, \\
	h_{i-1}(\mathbf{s}) & \text{if $s_i=\lr$}, \\
	h_{i-1}(\mathbf{s})-1 & \text{if $s_i=\sd$}, \\
	h_{i-1}(\mathbf{s}) & \text{if $s_{i}=\lb$},\\
	h_{i-1}(\mathbf{s}) & \text{if $s_i=\ly$}.\\	
	\end{array}
	\right.
	\end{equation}
	This implies \eqref{grouplag4} immediately.
\end{proof}
We define a  $\mathbb{Z}_2^n$-action on $\LH'_n$, which is  similar 
to Yan-Zhou-Lin's  group action on 
$\LH_n$ in \cite{YZL19} and  a generalization of Lin's group action on $\2MP^{^*}_n$ in \cite{Lin17}.
 Let $i\in[n]$ and $(\mathbf{s, p})\in\LH'_n$. Define the group action $\theta_i$ as follows, 
\[
\theta_i((\mathbf{s, p}))=
\left\{
\begin{array}{ll}
(\mathbf{s, p}) & \text{if $i\in \ly(\mathbf{s})$,} \\
(\mathbf{s', p}) & \text{otherwise,} 
\end{array}
\right.
\] 
where $\mathbf{s'}$ is the 3-Motzkin path  obtained from $\mathbf{s}$ by 
changing $s_i$ as
$L_b\leftrightarrow L_r$.
For any subset $S\subseteq [n]$ define the mapping $\theta'_{S}: \LH'_n\rightarrow \LH'_n$ by 
\begin{equation}\label{def:lagaction}
\theta'_S((\mathbf{s, p}))=\prod _{i\in S} \theta_i((\mathbf{s, p})).
\end{equation} 

Hence the group $\mathbb{Z}_2^n$ acts on $\LH'_n$ via the function $\theta'_S$.
Note  that  the three  sequences $\ly(\mathbf{s})$, 
$\mathbf{p}$ and $(h_{0}(\mathbf{s}), \ldots, (h_{n-1}(\mathbf{s}))$ are invariant
under the group action.
This action divides the set $\LH'_n$ into disjoint orbits and each orbit has a unique restricted Laguerre history whose  level steps are $\neq L_y$ or $\lb$.
For any fixed $(\mathbf{s, p})\in \LH'_n$ let $\Orb((\mathbf{s, p})):=\{\,\theta'_{S}((\mathbf{s, p}))\mid S\subseteq[n]\,\}$.
For $0\leq j\leq n$  we define
\begin{subequations}
\begin{align}
\S_{n, j}&=\{\sigma\in \S_n: \fix\sigma=j \},\label{def:Sj}\\
\mathcal{R}_{n,j}&=\{(\mathbf{s, p})\in\LH'_n : |\s|_{\ly}=j\}.\label{def:Rj}
\end{align}
\end{subequations}
where  $|\s|_a$ means the number of letters $a$ in the word $\s$. 
\begin{lem}
	Let $(\mathbf{s, p})\in\mathcal{R}_{n, j}$.  
	We have
	\begin{gather}
	(1+x)^{|\s|_{\lb}+|\s|_{\lr}}\sum_{(\mathbf{s', p})\in\Orb(\mathbf{s, p})}
	t^{|\s'|_{\lb}}\nonumber\\
	=\sum_{(\mathbf{s', p})\in\Orb(\mathbf{s, p})}(1+xt)^{|\s'|_{\lr}}
	(x+t)^{|\s'|_{\lb}}.\label{orbitc1}
	\end{gather}
\end{lem}
\begin{proof}
Let $L(\s)=\{i\in [n]: \textrm{$s_i=\lb$ or $s_i=\lr$}\}$	with cardinality $\ell=|\s|_{\lr}+|\s|_{\lb}$.
	By \eqref{def:lagaction}, the  group action  establishes a bijection from  the set of subsets of 
$L(\mathbf{s})$ to $\Orb(\mathbf{s, p})$ such that  if $S\subset L(\mathbf{s})$ then
$S=\lr\,(\mathbf{s'})$ with  $(\mathbf{s', p})=\theta'_{S}((\mathbf{s, p}))$. 
	Eq.~\eqref{orbitc1} is equivalent to
	$$
	(1+x)^\ell \sum_{S\subset L(\mathbf{s})}t^{|S|}=\sum_{S\subset L(\mathbf{s})}(1+xt)^{|S|}
	(x+t)^{\ell-|S|},
	$$
	namely,
	$(1+x)^\ell\cdot (1+t)^\ell=(1+xt+x+t)^\ell$.
\end{proof} 
\begin{proof}[Proof of Theorem~\ref{thm:cpeakS_1}]
For $(\mathbf{s, p})\in\mathcal{R}_{n, j}$, 
 since $2|\s|_{\su}+|\s|_{\lr}+|\s|_{\lb}+j=n$, by   \eqref{orbitc1}, we have 
\begin{gather}\label{eq:4.7}
(1+x)^{n-j-2|\s'|_{\su}}\sum_{(\mathbf{s', p})\in\Orb(\mathbf{s, p})}t^{|\s'|_{\lb}}\\
=\sum_{(\mathbf{s', p})\in\Orb(\mathbf{s, p})}(1+xt)^{n-j-2|\s'|_{\su}-|\s'|_{\lb}}(x+t)^{|\s'|_{\lb}},\nonumber
\end{gather}
that is,
\[
\sum_{(\mathbf{s', p})\in\Orb(\mathbf{s, p})}
t^{|\s'|_{\lb}}=\sum_{(\mathbf{s',\,p})\in\Orb(\mathbf{s, p})}\frac{(1+xt)^{n-j-2|\s'|_{\su}-|\s'|_{\lb}}(x+t)^{|\s'|_{\lb}}}{(1+x)^{n-j-2|\s'|_{\su}}}.
\]
Summing over all the orbits leads to 

\[
\sum_{(\mathbf{s, p})\in\mathcal{R}_{n, j}}t^{|\s|_{\lb}}=\sum_{(\mathbf{s, p})\in\mathcal{R}_{n, j}}\frac{(1+xt)^{n-j-2|\s|_{\su}-|\s|_{\lb}}(x+t)^{|\s|_{\lb}}}{(1+x)^{n-j-2|\s|_{\su}}}.
\]
Let
$$
|\mathbf{p}|=\sum_{i=1}^np_i\quad \textrm{and}\quad  h(\s)=\sum_{i=1}^nh_{i-1}(\mathbf{s}).
$$
Since $\su(\mathbf{s})$, $\mathbf{p}$ and $(h_0(\mathsf{s}), \ldots, 
h_{n-1}(\mathbf{s}))$ are invariant under the group action, 
\begin{align}\label{Laggroup1}
&\sum_{(\mathbf{s, p})\in\mathcal{R}_{n, j}}\biggl(p^{|\mathbf{p}|}q^{h(\s)-|\mathbf{p}|}\biggr)t^{|\s|_{\lb}+|\s|_{\su}}\nonumber\\
=&\sum_{(\mathbf{s, p})\in\mathcal{R}_{n, j}}\biggl(p^{|\mathbf{p}|}q^{h(\s)-|\mathbf{p}|}\biggr)\frac{(1+xt)^{n-j-2|\s|_{\su}-|\s|_{\lb}}(x+t)^{|\s|_{\lb}}t^{|\s|_{\su}}}{(1+x)^{n-j-2|\s|_{\su}}}.
\end{align}
By Lemma~\ref{lem: phiFZ}, as the bijection $\phi'_{FZ}$ maps $\S_{n, j}$ to $\mathcal{R}_{n, j}$ with corresponding statistics, we can rewrite \eqref{Laggroup1} as
\begin{align}
&\sum_{\sigma\in\S_{n, j}}\biggl(p^{\nest\,\sigma}q^{\cros\,\sigma+\exc\,\sigma}\biggr)t^{\exc\,\sigma}\label{lageq}\\
=&\sum_{\sigma\in\S_{n, j}}
\biggl(p^{\nest\,\sigma}q^{\cros\,\sigma+\exc\,\sigma}\biggr)
\frac{(1+xt)^{n-j-\exc\,\sigma-\cpeak\,\sigma}(x+t)^{\exc\,\sigma-\cpeak\,\sigma}t^{\cpeak\,\sigma}}{(1+x)^{n-j-2\cpeak\,\sigma}}\nonumber\\
=&
\left(\frac{1+xt}{1+x}\right)^{n-j}P^{(\nest, \cros, \cpeak,\exc)}\left(\S_{n, j}; p, q, \frac{(1+x)^{2}t}{(x+t)(1+xt)},\frac{q(x+t)}{1+xt}\right).\nonumber
\end{align}
Multiplying \eqref{lageq}  by $r^j$ and summing over $j$  yields 
 (\ref{cpeakS_1}). By using
the substitution $u=\frac{(1+x^2)t}{(x+t)(1+xt)}$ and $v=\frac{x+t}{1+xt}$ as in \eqref{cpeakS_1},  we obtain (\ref{cpeakS_2}) immediately.
\end{proof}


\begin{remark} We show that  Eq.~\eqref{Laggroup1} implies  also  two other known results in the litterature.
	When $x=1$ Eq.\eqref{Laggroup1} reduces to
	\begin{align}\label{Laggroup2}
	\sum_{(\mathbf{s, p})\in\mathcal{R}_{n, j}}\biggl(p^{|\mathbf{p}|}q^{h(\s)-|\mathbf{p}|}\biggr)t^{|\s|_{\lb}+|\s|_{\su}}=
	\sum_{(\mathbf{s, p})\in\mathcal{R}_{n, j}}\biggl(p^{|\mathbf{p}|}q^{h(\s)-|\mathbf{p}|}\biggr)\frac{(1+t)^{n-j-2|\s|_{\su}}t^{|\s|_{\su}}}{2^{n-j-2|\s|_{\su}}}.
	\end{align}
Let
	\begin{align*}
	\mathcal{O}_{n,k, j}=\{(\mathbf{s, p})\in \mathcal{R}_{n, j}: 
	|\s|_{\lb}=0\;\textrm{ and }\; |\s|_{\su}=k\}.
	\end{align*}
	By the group action on $\mathcal{R}_{n, j}$, we see that there are $2^{n-j-2|\s|_{\su}}$ elements in each orbit,  and then
	$$
	2^{n-2k-j}|\mathcal{O}_{n, k,j}|=|\{(\mathbf{s, p})\in\mathcal{R}_{n, j}: |\s|_{\su}=k\}|.
	$$
	Hence \eqref{Laggroup2} is equivalent to 
	\begin{align}\label{Laggroup3}
	\sum_{(\mathbf{s, p})\in\mathcal{R}_{n, j}}\biggl(p^{|\mathbf{p}|}q^{h(\s)-|\mathbf{p}|}\biggr)t^{|\s|_{\lr}+|\s|_{\su}}=\sum_{k=0}^n\sum_{(\mathbf{s, p})\in\mathcal{O}_{n, j, k}}\biggl(p^{|\mathbf{p}|}q^{h(\s)-|\mathbf{p}|}\biggr)(1+t)^{n-j-2k}t^k.
	\end{align} 
Thanks to  the bijection $\phi'_{FZ}$ and  \eqref{grouplag1}--\eqref{grouplag4} we obtain Theorem~8 in \cite{SZ12},
	\begin{align}
	&\sum_{\pi\in\S_n}(tq)^{\exc\,\pi}p^{\nest\,\pi}q^{\cros\,\pi}r^{\fix\,\pi}\nonumber\\
	=&
	\sum_{j=0}^{n}r^j\sum_{k=0}^{\lfloor(n-j)/2\rfloor}\biggl(
	\sum_{\sigma\in\S_{n,k, j}}p^{\nest\,\sigma}q^{\cros\,\sigma+\exc\,\sigma}
	\biggr)t^k(1+t)^{n-j-2k},\label{SZTh8}
	\end{align}
	where 
	$$
	\S_{n, k,j}=\{\sigma\in\S_n, \cpk\,\sigma= k, \fix\,\sigma= j, \cda\,\sigma=0\}.
	$$ 	
	By Lemma~\ref{321:nest} and \eqref{INV},
	letting $p=0$ in \eqref{SZTh8}  yields  Theorem 2.4 in 
 \cite{Lin17}, 
	\begin{align}\label{linfixversion}
	&\sum_{\sigma\in\S_n(321)}t^{\exc\,\sigma}q^{\inv\,\sigma}r^{\fix\,\sigma}\nonumber\\
	=&\sum_{j=0}^nr^j\sum_{k=0}^{\lfloor {(n-j)}/{2}\rfloor}\left(\sum_{\sigma\in\S_{n, k,j}(321)}q^{\inv\,\sigma}\right)t^k(1+t)^{n-j-2k},
	\end{align}
	where 
	$$\S_{n, k,j}(321):=\{\sigma\in\S_{n}(321):\fix\,\sigma=j,\, \exc\,\sigma=k,~\cda\,\sigma=0\}.$$
	Note that when $r=1+t$, Eq.~\eqref{linfixversion} reduces to Eq.~\eqref{general-qgamma} with the $\gamma$-coefficients in \eqref{eq:inv321}--\eqref{eq:213:132}.
\end{remark}

\section{Proofs via continued fractions}

For convenience, 
we use the following compact notation for  the J-type continued fraction
\begin{align}\label{J-CF}
J[z; b_n, \lambda_n]=
\cfrac{1}{1-b_0z-
\cfrac{\lambda_1z^2}{1-b_1z-
\cfrac{\lambda_2\,z^2}{1-b_2z-
\cfrac{\lambda_3z^2}{1-\cdots}}}}.
\end{align}
We shall use the notation  $[n]_{p,q}:=(p^n-q^n)/(p-q)$ for $n\in \N$.
\subsection{Some combinatorial continued fractions}

We first recall a standard contraction formula for continued fractions, see \cite[Eq. (44)]{SZ10}.
\begin{lem}[Contraction formula]\label{contra-formula}
\,\,The following contraction formulae  hold
\begin{align*}
\cfrac{1}{
1-\cfrac{\alpha_{1}z}{1-\cfrac{\alpha_{2}z}{
1-\cfrac{\alpha_{3}z}{
1-\cfrac{\alpha_{4}z}
{1-\cdots}}}}}
&=\cfrac{1}{1-\alpha_{1}z-\cfrac{\alpha_{1}\alpha_{2}z^{2}}{1-(\alpha_{2}+\alpha_{3})z-\cfrac{\alpha_{3}\alpha_{4}z^{2}}{1-\cdots}}}\\
&=1+\cfrac{\alpha_{1}z}{1-(\alpha_{1}+\alpha_{2})z-\cfrac{\alpha_{2}\alpha_{3}z^{2}}{1-(\alpha_{3}+\alpha_{4})z-\cfrac{\alpha_{4}\alpha_{5}z^{2}}{1-\cdots}}}.
\end{align*}
\end{lem}
 The following  four combinatorial continued fraction expansions are proved by  
Shin and Zeng~\cite{SZ12}.
 Let
\begin{subequations}
\begin{align}
A_n(p,q,t,u,v,w) := &\sum_{\sigma\in \S_n} p^{\nest\, \sigma} q^{\cros \,\sigma}  t^{\exc\,\sigma}u^{\cdd^*\,\sigma} v^{\cda^*+\fix^*\,\sigma}w^{\cpk^*\,\sigma}\label{C1eq:dfA}\\
=& \sum_{\sigma\in \S_n} p^{(\bac)\, \sigma} q^{(\cab)\, \sigma}  t^{\des\,\sigma}u^{\da\,\sigma} v^{\dd\,\sigma}w^{\peak\,\sigma-1},\label{C1eq:dfA1}
\end{align}
\end{subequations}
where the equality of the two enumerative polynomials  follows from Lemma~\eqref{starpqzhuang}.

\begin{lem}\cite[Eq. (28)]{SZ12}\label{Cha1generalA}
We have
\begin{subequations}
\begin{equation}\label{C1eq:A}
\sum_{n\geq0} A_{n+1}(p,q,t,u,v,w)z^{n}=J[z; b_n, \lambda_n]
\end{equation}
with
\begin{align}
b_n=(u+t v)[n+1]_{p,q},\quad
\lam_n=[n]_{p,q}[n+1]_{p,q}.
\end{align}
\end{subequations}
\end{lem}

 Let
\begin{align}
B_n(p,q,t,u,v,w,y) :=\sum_{\sigma\,\in \S_n} p^{\nest\, \sigma} q^{\cros \,\sigma} t^{\exc\,\sigma} u^{\cdd\,\sigma} v^{\cda\,\sigma} w^{\cvalley\,\sigma}y^{\fix\,\sigma}.\label{SZ12eq2B}
\end{align}
\begin{lem}\cite[Eq. (34)]{SZ12}\label{SZ12Bset}
We have 
\begin{subequations}
\begin{align}\label{cf-septuple}
1+\sum_{n=1}^{\infty}B_n(p, q, t, u, v, w, y)z^n=J[z; b_n, \lambda_n],
\end{align}
with
\begin{align}
 b_n=yp^n+(qu+tv)[n]_{p, q}, \quad
\lambda_n=tw[n]^2_{p, q}.
\end{align} 
\end{subequations}
\end{lem}

Let
\begin{align} \label{C1SZ12eqcycfix}
C_n(q, t, u, v, w) :=
\sum_{\sigma\in \mathfrak{S}_n} q^{\cyc^*\,\sigma-\fix^*\,\sigma}  t^{\wex^*\,\sigma} u^{\cda^*+\fix^*\,\sigma} v^{\cdd^*\,\sigma} w^{\cvalley^*\,\sigma}.
\end{align}

\begin{lem}\cite[Eq. (50)]{SZ12}\label{SZ12cycfix}
We have
\begin{subequations}
\begin{equation}\label{C1Acf-quintuple}
\sum_{n\geq 0}^{\infty}C_{n+1}(q, t, u, v, w)z^{n}=J[z; b_n, \lambda_n]
\end{equation}
with
\begin{align}
 b_n=(n+1)(tu+v), \quad
\lambda_n=n(q+n)tw.
\end{align} 
\end{subequations}

\end{lem}
Let
\begin{align} \label{C1SZ12eqcyc}
D_n(q, t, u, v, w) :=
\sum_{\sigma\in \mathfrak{D}_n} q^{\cyc\,\sigma}  t^{\exc\,\sigma} u^{\cda\,\sigma} v^{\cdd\,\sigma} w^{\cvalley\,\sigma}.
\end{align}
\begin{lem}\cite[Eq. (41)]{SZ12}\label{SZ12cycset}
We have
\begin{subequations}
\begin{equation}\label{C1cf-quintuple}
1+\sum_{n=1}^{\infty}D_n(q, t, u, v, w)z^n=J[z; b_n, \lambda_n]
\end{equation}
with
\begin{align}
 b_n=n(tu+v), \quad
\lambda_n=n(q+n-1)tw.
\end{align} 
\end{subequations}
\end{lem}

\subsection{Proof of Theorems~\ref{pqpk}, \ref{t-Aqcpkexc}--\ref{thm:cpeakS_1}, \ref{qnara}--\ref{qwexth}}
In the previous section Theorems~3.1 and 3.7 are proved using group actions. Here
we shall give an alternative proof for Theorems~3.1 and 3.7 using continued fractions.

\begin{proof}[Proof of Theorem \ref{pqpk}]
In view of \eqref{C1eq:A},
 we have 
$$
A_n(p, q, t, 1, 1, x)= \sum_{\sigma\in{\S}_n}p^{\nest\,\sigma}q^{\cros\,\sigma}t^{\exc\,\sigma}x^{\cpeak^*\,\sigma}.
$$
It follows that 
\begin{subequations}
\begin{align}\label{rightdes}
\sum_{n=0}^{\infty}A_{n+1}(p, q, t, 1, 1, x)z^n=J[z;b_n,\lam_n]
\end{align}
with
\begin{align}
 b_n=[n+1]_{p, q}(t+1), \quad
\lambda_n=[n]_{p, q}[n+1]_{p,q} tx.
\end{align} 
\end{subequations}
When $x=1$ we have the J-fraction for $\sum_{n=0}^{\infty}A_{n+1}(p, q, t)z^n$.
The g.f. of the right side of Eq. \eqref{eqpk1} is
\begin{align*}
\sum_{n\geq 0}P^{(\nest, \cros, \cpeak^*,\exc)}\left(\S_{n+1}; p, q, \frac{(1+x)^2t}{(x+t)(1+xt)},\frac{x+t}{1+xt}\right)\left(\frac{(1+xt)}{1+x}\right)^{n}z^n.
 \end{align*}
Substituting $t, x$ and  $z$,  respectively, 
 by 
 $$
 \frac{x+t}{1+xt}, \;\; \frac{(1+x)^2t}{(x+t)(1+xt)}\;\;\textrm{and}\; \; \frac{(1+xt)z}{1+x}
 $$
  in \eqref{rightdes}, we obtain the J-fraction of  $\sum_{n=0}^{\infty}A_{n+1}(p, q, t)z^n$.
\end{proof}

\begin{proof}[Proof of Theorem \ref{t-Aqcpkexc}]
By  Eq. \eqref{C1SZ12eqcycfix},
the g.f. of the left side of Eq.  \eqref{Aqexcpk1} is
\begin{subequations}
\begin{equation}\label{leftcycfix}
\sum_{n\geq0}^{\infty}A_{n+1}^{(\cyc^*-\fix^*, \exc)}(q, t)z^{n}=J[z;b_n,\lam_n]
\end{equation}
with
\begin{align}
 b_n=(n+1)(t+1), \quad
\lambda_n=n(q+n).
\end{align} 
\end{subequations}
By definition, the g.f. of the right side of Eq.  \eqref{Aqexcpk1} is
\begin{align*}
\sum_{n\geq 0}P^{(\cyc^*-\fix^*, \cval^*,\exc)}\left(\S_{n+1}; q, \frac{(1+x)^2t}{(x+t)(1+xt)},\frac{x+t}{1+xt}\right)\left(\frac{(1+xt)}{1+x}\right)^{n}z^n.
 \end{align*}
Letting $u=v=1$ and transforming $\frac{x+t}{1+xt}$, $\frac{(1+x)^2t}{(x+t)(1+xt)}$ and $\frac{z(1+xt)}{(1+x)}$ to $t$, $x$ and $z$ in \eqref{C1Acf-quintuple}, respectively, we obtain \eqref{leftcycfix}  immediately. 
%
\end{proof}

\begin{proof}[Proof of Theorem \ref{thm:SZ3q}]


Letting $u=v=w=1$ in  Eq. \eqref{C1cf-quintuple},
we see that 
the g.f. of the left side of Eq. \eqref{exc3qpk1} is
\begin{subequations}
\begin{equation}\label{leftcyc}
1+\sum_{n=1}^{\infty}D_n^{(\cyc, \exc)}(q, t)z^n
=J[z;b_n,\lam_n]
\end{equation}
with
\begin{align}
 b_n=n(t+1), \quad
\lambda_n=n(q+n-1)t.
\end{align} 
\end{subequations}

On the other hand, the g.f. of the right side of Eq.  \eqref{exc3qpk1} is
\begin{align*}
\sum_{n\geq 0}P^{(\cyc, \cpeak,\exc)}\left(\D_n; q, \frac{(1+w)^2t}{(w+t)(1+wt)},\frac{w+t}{1+wt}\right)\left(\frac{(1+wt)}{1+w}\right)^{n}z^n.
 \end{align*}
 Letting $u=v=1$ in  \eqref{C1cf-quintuple} and substituting 
$t$, $w$ and $z$ with 
$$
\frac{w+t}{1+wt}, \quad \frac{(1+w)^2t}{(w+t)(1+wt)}\quad \textrm{and} \quad \frac{(1+wt)z}{(1+w)},
$$
 respectively, we obtain the J-fraction in \eqref{leftcyc}  immediately. 
%
\end{proof}

\begin{proof}[Proof of Theorem \ref{thm:cpeakS_1}]
We prove that both sides of Eq. \eqref{cpeakS_1} have the same g.f. by comparing their continued fraction expansions.
By \eqref{SZ12eq2B}, we have 
 $$
B_n(p, q, qt, 1, 1, 1, r)=P^{(\nest, \cros,  \exc, \fix)}(\S_n; p, q, tq, r).
$$
It follows from Eq.~\eqref{cf-septuple} that 
\begin{subequations}
\begin{align}\label{eq3.2:left}
\sum_{n\geq 0}P^{(\nest, \cros,  \exc, \fix)}(\S_n; p, q, tq, r)z^n=J[z; b_n, \lambda_n]
\end{align}
with
\begin{align}
 b_n=rp^n+q(1+t)[n]_{p, q}, \quad
\lambda_n=tq[n]^2_{p, q}.
\end{align} 
\end{subequations}
On the other hand, by definition and invoking the equality  $\cpk=\cval$ (cf. \eqref{cpk=cval}),
the g.f. of the right-hand side of Eq. \eqref{cpeakS_1} is
\begin{align}\label{eq3.2:right}
\sum_{n\geq 0}\left(\frac{1+xt}{1+x}\right)^{n}P^{(\nest, \icr, \cpeak, \exc, \fix)}\left(\S_n; p, q, \frac{(1+x)^{2}t}{(x+t)(1+xt)},\frac{q(x+t)}{1+xt}, \frac{(1+x)r}{1+xt}\right)z^n.
\end{align}
In Eq.~\eqref{cf-septuple}  letting $u=v=1$ and making the substitution
$$
t\leftarrow \frac{q(x+t)}{1+xt}, \; \;w\leftarrow \frac{(1+x)^2t}{(x+t)(1+xt)}, \;\;y\leftarrow \frac{(1+x)r}{1+xt},\; \;
z\leftarrow \frac{(1+xt)z}{(1+x)},
$$
we see that the g.f. \eqref{eq3.2:right} has the same J-fraction expansion as 
\eqref{eq3.2:left}.
\end{proof}

\begin{proof}[Proof of Theorem \ref{qnara}]
Recall~\cite[Theorem 7.2]{CEKS13}  that
\begin{align}\label{general-CEKS}
\sum_{n=0}^{\infty} N_n(t,q, r)z^n
=\cfrac{1}{1-rz-\cfrac{tqz^2}{1-(1+t)qz-\cfrac{tq^3z^2}{1-(1+t)q^2z-\cfrac{tq^5z^2}{\ddots}}}}.
\end{align}
Writing $\sum_{n=1}^{\infty} N_{n-1}(t,q, r)z^n=z\sum_{n=0}^{\infty} N_n(t,q, r)z^n$ we have 
\begin{align}\label{nararecu}
1+\sum_{n\geq 1}N_{n-1}(t,q, 1+t)z^n=
1+\cfrac{z}{1-(1+t)z-\cfrac{tqz^{2}}{1-q(t+1)z-\cfrac{tq^3z^{2}}{1-\cdots}}},
\end{align}
which  is  
$\sum_{n\geq 0}N_n(t/q, q,1)z^n$ by applying  Lemma~\ref{contra-formula}. Thus 
\begin{align}\label{key1}
N_n(t/q,q,1)=N_{n-1}(t,q, 1+t).
\end{align}
By \eqref{nest=nest-1} and 
and  Lemma~\ref{321:nest}, 
we see 
that  $\pi\in \S_n(321)$ if and only if $ \pi^{-1}\in \S_n(321)$.
As $\wex\,\pi^{-1}=n-\drop\,\pi^{-1}=n-\exc\pi$ and 
$\inv\, \pi^{-1}=\inv\,\pi$ we have  
\begin{align*}
N_n(t,q,t)&=\sum_{\pi\in\S_n(321)}t^{\wex\,\pi^{-1}}q^{\inv\,\pi^{-1}}\\
&=t^n\sum_{\pi\in\S_n(321)}(1/t)^{\exc\,\pi}q^{\inv\,\pi}=t^nN_n(1/t, q, 1).
\end{align*}
It follows from \eqref{key1} that 
\begin{align}\label{key2}
N_n(t,q,t)=t^nN_{n-1}({q}/t, q, 1+{q}/t).
\end{align}
In view of \eqref{def:N} identities~\eqref{key1} and  \eqref{key2} provide 
 the first interpretation  in Table~1. Other   interpretations in Table~1  can be derived
 from the  
equidistribution results  in  \eqref{eq:key} and \eqref{patternres1}--\eqref{patternres4}

\end{proof}


\begin{proof}[Proof of Theorem \ref{thm:general-qgamma}]
By Lemma~\ref{321:nest}, \eqref{INV} and \eqref{SZ12eq2B}, we have
\begin{align}\label{pateq1}
B_n(0, q, tq, 1, 0, 1, 1)= \sum_{\sigma\in\widetilde{\S}_{n}(321)}q^{\inv\,\sigma}t^{\exc\,\sigma}.
\end{align}
It follows from Lemma~\ref{SZ12Bset} that
\begin{align}\label{pateq2}
\sum_{n=0}^{\infty}\sum_{\sigma\in\widetilde{\S}_{n}(321)}q^{\inv\,\sigma}t^{\exc\,\sigma}z^n=\cfrac{1}{1-z-\cfrac{tqz^2}{1-qz-\cfrac{tq^3z^2}{1-q^2z-\cfrac{tq^5z^2}{\ddots}}}}.
\end{align}
Now, the g.f. of the right-hand  side of Eq. \eqref{general-qgamma} is 
\begin{align}
GF:&=1+z\sum_{n\geq 1}\sum_{k=0}^{\lfloor\frac{n-1}{2}\rfloor}\left(\sum_{\sigma\in\widetilde{\S}_{n-1,k}(321)}q^{\inv\,\sigma}\right)t^k(1+t)^{n-1-2k}z^{n-1}\nonumber\\
=&1+z\sum_{n\geq 0}\sum_{\sigma\in\widetilde{\S}_{n}(321)}q^{\inv\,\sigma}\left(\frac{t}{(1+t)^2}\right)^{\exc\,\sigma}((1+t)z)^n.\nonumber
\end{align}
By \eqref{pateq2} and  Lemma~\ref{contra-formula} we see that 
\begin{align*}
GF=1&+\cfrac{z}{1-(1+t)z-\cfrac{tqz^2}{1-(1+t)qz-\cfrac{tq^3z^2}{1-(1+t)q^2z-\cfrac{tq^5z^2}{\ddots}}}}\\
&=
\dfrac{1}{
1 -z- \dfrac{t z^2}{
1 - (q+t) z - \dfrac{tq^2 z^2}{
1 - (q+t)q z - \dfrac{tq^4 z^2}{
\cdots}}}},
\end{align*} 
which is equal to $\sum\limits_{n\geq 0} N_n(t/q, q, 1)z^n$ by  \eqref{general-CEKS}.  Other interpretations can be obtained by the 
equidistribution results of \eqref{eq:key} and \eqref{patternres1}--\eqref{patternres4}.
\end{proof}
\begin{proof}[Proof of Theorem \ref{qwexth}]

The g.f. of the right side of \eqref{narawex} can be written as
\begin{align}
&1+z\sum_{n\geq 1}\sum_{k=1}^{\lfloor\frac{n+1}{2}\rfloor}\left(\sum_{\sigma\in\widetilde{\S}_{n-1,k-1}(321)}q^{n-1+\inv\,\sigma-\exc\,\sigma}\right)t^k(1+t/q)^{n+1-2k}z^{n-1}\nonumber\\
=&1+zt\sum_{n\geq 0}\left(\sum_{\sigma\in\widetilde{\S}_{n}(321)}q^{\inv\,\sigma-\exc\,\sigma}\left(\frac{tq^2}{(q+t)^2}\right)^{\exc\,\sigma}\right)((q+t)z)^n.\label{lin2}
\end{align}
By using the second claim of Lemma \ref{321:nest}, Eq.\eqref{INV} and 
\eqref{SZ12eq2B}, we have
$$B_n(0, q, t, 1, 0, 1, 1)= \sum_{\sigma\in\widetilde{\S}_{n}(321)}q^{\inv\,\sigma-\exc\,\sigma}t^{\exc\,\sigma},$$
Lemma \ref{SZ12Bset} implies that 
\begin{align*}
\sum_{n=0}^{\infty}\sum_{\sigma\in\widetilde{\S}_{n}(321)}q^{\inv\,\sigma-\exc\,\sigma}t^{\exc\,\sigma}z^n=\cfrac{1}{1-z-\cfrac{tz^2}{1-qz-\cfrac{tq^2z^2}{1-q^2z-\cfrac{tq^4z^2}{\ddots}}}}.
\end{align*}
Making the substitution 
 $z\mapsto (q+t)z$ and $t\mapsto {tq^2}/{(q+t)^2}$ 
  in the above equation and applying the contarction formulae,
 we obtain
\begin{align}\label{rqwex}
1&+\cfrac{zt}{1-(q+t)z-\cfrac{tq^2z^2}{1-(q+t)qz-\cfrac{tq^4z^2}{1-(q+t)q^2z-\cfrac{tq^6z^2}{\ddots}}}}\\
&=\dfrac{1}{
1 -tz- \dfrac{tqz^2}{
1 - (1+t) qz - \dfrac{tq^3 z^2}{
1 - (1+t)q^2z - \dfrac{tq^5 z^2}{
\cdots}}}},
\end{align}
 which is equal to $\sum\limits_{n\geq 0} N_n(t, q, t)z^n$ by \eqref{general-CEKS}.
 Other interpretations can be obtained by the 
equidistribution results of \eqref{eq:key} and \eqref{patternres1}--\eqref{patternres4}.
\end{proof}

\subsection{Proof of Theorems~\ref{thmBgfexc} and \ref{Bnegexc}}

Recall  the \emph{color order} $<_{c}$ of $\{\pm 1,\ldots, \pm n\}$:
$$
-1<_{c}-2<_{c}\cdots <_{c}-n<_{c}1<_{c}2<_{c}\cdots <_{c}n,
$$
and define the following statistics: 
\begin{align*}
\fix\,\sigma &= \#\{i \in [n]: i = \sigma(i)\},\\
\exca\sigma &= \#\{i\in [n]: i<_c \sigma(i)\},\\
\wexa \sigma &= \#\{i\in [n]: i\le_c \sigma(i)\} = \exca \sigma + \fix \sigma,\\
\wexc \sigma &= \#\{i\in [n]: i\leq|{\sigma(i)}| \text{ and } \sigma(i)<0 \},\\
\neg\,\sigma&=\#\{i\in[n]: \sigma_{i}<0\}.
\end{align*}
Let 
$$
F_n(q, t,w,r,y)=\sum_{\sigma \in \B_n} q^{\cros\,\sigma} t^{\wexa\sigma} w^{\wexc\sigma} r^{\fix\,\sigma} y^{\neg\,\sigma}.
$$
The following result is the $r=2$ case of \cite[Lemma 16]{SZ16}.
\begin{lem}
We have 
\begin{align}\label{CF-B}
\sum_{n\ge 0} F_n(q, t,w,r,y) z^n=J[z; b_n, \lambda_n],
\end{align}
with 
\begin{align*}
\lambda_n &= (t + wyq^{n-1}) (1 + yq^{n}) [n]_q^2, \\
b_n &= (1 + yq^n)[n]_q + t( r + q[n]_q) + wyq^n [n+1]_q.
\end{align*}
\end{lem}
%


We need the following lemma, see \cite[Lemma 12]{Zeng93} and 
\cite[p. 307]{GJ83}.
\begin{lem}\label{binomialtransform} If  two sequences 
$\{\mu_n\}_n$ and $\{\nu_n\}_n$ satisfy the equation

$$
\sum_{n\geq 0} \mu_n\frac{z^n}{n!}=e^{\alpha z}\sum_{n\geq 0} \nu_n\frac{z^n}{n!},
$$
then
$$
\sum\limits_{n\geq 0} \nu_n z^n=J[z; b_n, \lambda_n]\Longrightarrow 
\sum\limits_{n\geq 0} \mu_n z^n=J[z; b_n+\alpha, \lambda_n].
$$
\end{lem}
\begin{proof}[Proof of Theorem~\ref{thmBgfexc}]
Since 
 $\exc=(\wexa-\fix)+\wexc$, see \cite[Eq. (4.5)]{SZ16},
we have 
$$
F_n(1, t,t,1/t,y)
=\sum_{\sigma\in\B_n}t^{\exc_B\sigma} y^{\neg\,\sigma}
$$
and  formula \eqref{CF-B} becomes
\begin{align}\label{Bleft}
\sum_{n\geq 0}\left(\sum_{\sigma\in\B_n}t^{\exc_B\sigma} y^{\neg\,\sigma}\right)z^n=J[z;b_n,\lambda_n],
\end{align}
where $b_n=(n+1)(1+yt)+n(t+y)$ and $\lambda_n=n^2\,(1+y)^2t$.
 By Lemma~\ref{binomialtransform} we derive  from \eqref{lag2} and \eqref{brenti} that $\sum_{n\ge 0}B_n(y,t) z^n$ has the same continued fraction expansion in \eqref{Bleft}. 
\end{proof}

%
%

%
%

\begin{proof}[Proof of Theorem~\ref{Bnegexc}]
By \eqref{SZ12eq2B}, we have 
$$
P^{(\cpeak, \exc)}(\S_n; w, t)=B_n(1, 1, t, 1, 1, w, 1).
$$
Specializing $(p, q, t, u, v, w, y)$ in Lemma~\ref{SZ12Bset} yields
 \begin{align*}
\sum_{n=0}^{\infty}P^{(\cpeak, \exc)}(\S_n; w, t)z^n=J[z; b_n,\lambda_n],
\end{align*}
where $ b_n=1+n\,(1+t)$ and  $\lambda_n=n^2\,tw$.
It follows that the series
\begin{align*}
\sum_{n\geq 0}P^{(\cpeak, \exc)}\left(\S_n;\frac{(1+y)^2t}{(y+t)(1+yt)}, \frac{y+t}{1+yt}\right)((1+yt)z)^n
 \end{align*}
 has the same continued fraction expansion for $\sum_{n\geq 0}B_n(y, t)z^n$ in \eqref{Bleft}.
 \end{proof}

\section*{Acknowledgement}
The first two named authors were supported by the China Scholarship Council.
The second author's  work was done during his visit at Institut Camille Jordan, Universit\'e Claude Bernard Lyon 1 in 2018-2019.  The first author was supported by the Israel Science Foundation (grant no. 1970/18) during the revision of this work.

An extended abstract of this paper will appear in the Proceedings of the 32nd Conference on
Formal Power Series and Algebraic Combinatorics (FPSAC'20).

\end{document}